\documentclass[10pt,english]{article}
\usepackage[utf8]{inputenc}
\usepackage[a4paper]{geometry}
\usepackage{babel}
\usepackage{mathtools}
\usepackage{dsfont}
\usepackage{amsmath}
\usepackage{amsthm}
\usepackage{amssymb}
\usepackage{stmaryrd}
\usepackage[bookmarks=true,bookmarksnumbered=false,bookmarksopen=false,
 breaklinks=false,pdfborder={0 0 1},backref=false,colorlinks=false]
 {hyperref}
\hypersetup{
 linkcolor=blue}

\makeatletter
\numberwithin{equation}{section}
\numberwithin{figure}{section}
\theoremstyle{plain}
\newtheorem{thm}{\protect\theoremname}[section]
\theoremstyle{remark}
\newtheorem{rem}[thm]{\protect\remarkname}
\theoremstyle{plain}
\newtheorem{lem}[thm]{\protect\lemmaname}
\theoremstyle{plain}
\newtheorem{prop}[thm]{\protect\propositionname}

\usepackage{ dsfont }

\newcommand{\df}{\mathrm{d}}

\allowdisplaybreaks

\makeatother

\providecommand{\lemmaname}{Lemma}
\providecommand{\propositionname}{Proposition}
\providecommand{\remarkname}{Remark}
\providecommand{\theoremname}{Theorem}

\begin{document}
\global\long\def\df{\mathrm{def}}%
\global\long\def\eqdf{\stackrel{\df}{=}}%
\global\long\def\ep{\varepsilon}%
\global\long\def\ind{\mathds{1}}%

\title{Large-$n$ asymptotics for Weil-Petersson volumes of moduli spaces
of bordered hyperbolic surfaces }
\author{Will Hide and Joe Thomas}
\maketitle
\begin{abstract}
We study the geometry and spectral theory of Weil-Petersson random
surfaces with genus-$g$ and $n$ cusps in the large-$n$ limit.

We show that for a random hyperbolic surface in $\mathcal{M}_{g,n}$
with $n$ large, the number of small Laplacian eigenvalues is linear
in $n$ with high probability. By work of Otal and Rosas \cite{Ot.Ro09},
this result is optimal up to a multiplicative constant. 

We also study the relative frequency of simple and non-simple closed
geodesics, showing that on random surfaces with many cusps, most closed
geodesics with lengths up to $\log(n)$ scales are non-simple.

Our main technical contribution is a novel large-$n$ asymptotic formula
for the Weil-Petersson volume $V_{g,n}\left(\ell_{1},\dots,\ell_{k}\right)$
of the moduli space $\mathcal{M}_{g,n}\left(\ell_{1},\dots,\ell_{k}\right)$
of genus-$g$ hyperbolic surfaces with $k$ geodesic boundary components
and $n-k$ cusps with $k$ fixed, building on work of Manin and Zograf
\cite{Ma.Zo00}. 

{\footnotesize\tableofcontents{}}{\footnotesize\par}
\end{abstract}

\section{Introduction}

Let $\mathcal{M}_{g,n}$ denote the moduli space of genus-$g$ hyperbolic
surfaces with $n$ cusps. One can equip $\mathcal{M}_{g,n}$ with
a natutral probability measure $\mathbb{P}_{g,n}$ by normalising
the Weil-Petersson volume form. Stemming from the works of Guth, Parlier
and Young \cite{Gu.Pa.Yo11}, and Mirzakhani \cite{Mi2013} there
has been significant interest in the geometry and spectral theory
of Weil-Petersson random surfaces of large volume. By the Gauss-Bonnet
theorem, sampling a large volume hyperbolic surface corresponds to
taking $g+n\to\infty$. 

So far, the vast majority of attention in the literature has focused
on compact surfaces in the large-genus limit, $g\to\infty$, with
great success in understanding their typical geometric and spectral
features. On the other hand, the low genus setting has seen a reemergence
in the physics literature in relation to the study of Jackiw-Teitelboim
(JT) gravity (see the recent survey \cite{Me.Tu2023}), a simple solvable
model of quantum gravity in two dimensions. Moreover, there have been
recent connections \cite{Bu.Cu2024} drawn between Weil-Petersson
random surfaces in the $n\to\infty$ regime and random planar maps
which have been fruitful in investigating local and global geometric
properties of genus zero surfaces (see Section \ref{subsec:Related-work}).
These developments give strong impetus to study Weil-Petersson random
surfaces with many cusps and is the focus of the current paper. 

Prior work has shown that such surfaces exhibit very different and
interesting spectral geometric behaviours \cite{Zo93,Hi.Th.22,Sh.Wu22}
to their compact, large genus counterparts, but they are yet to be
intensively studied. The results obtained here focus on eigenvalues
of the Laplacian and structure of closed geodesics and are proven
through novel asymptotics for Weil-Petersson volumes of moduli spaces
of bordered hyperbolic surfaces with many cusps that will be applicable
to other spectral geometric questions in this setting.

\subsection{Small eigenvalues\protect\label{subsec:Applications1}}

Let $X$ be a complete, connected and orientable finite-area hyperbolic
surface. The spectrum of the Laplacian $\Delta_{X}$ is contained
inside $[0,\infty)$ with $0$ a simple eigenvalue. When $X$ is non-compact,
the spectrum in $(0,\frac{1}{4})$ consists of finitely many discrete
eigenvalues called \textit{exceptional eigenvalues} and absolutely
continuous spectrum in $[\frac{1}{4},\infty)$, with possibly infinitely
many embedded eigenvalues. The first result of this paper is concerned
with the number of exceptional eigenvalues, $N^{\text{exc}}\left(X\right)$,
for random surfaces in $\mathcal{M}_{g,n}$ as $n\to\infty$.

A result of Zograf \cite{Zo1987} states that for $X\in\mathcal{\mathcal{M}}_{g,n}$,
the infimum $\lambda_{1}(X)$ of the non-zero spectrum of the Laplacian
satisfies

\begin{equation}
\lambda_{1}\left(X\right)\leqslant C\frac{g+1}{n},\label{eq:zograf-lambda1}
\end{equation}
for a universal constant $C>0$. On the other hand, by a result of
Otal and Rosas \cite{Ot.Ro09}, any surface in $\mathcal{M}_{g,n}$
has at most $2g+n-3$ exceptional eigenvalues. Thus for fixed $g$
and sufficiently large $n$,
\begin{equation}
1\leqslant N^{\text{exc}}\left(X\right)\leqslant2g+n-3.\label{eq:exceptional-bound}
\end{equation}

Our first result says that for fixed $g$, a random surface in $\mathcal{M}_{g,n}$
saturates the upper bound (\ref{eq:exceptional-bound}) up to a constant
multiplicative factor, with probability tending to $1$ as $n\to\infty$.
\begin{thm}
\label{thm:eigenvalues}For any $\ep>0$ there exists a constant $C(\ep)>0$
such that for $g\geqslant0$ a Weil-Petersson random surface $X\in\mathcal{M}_{g,n}$
has at least $C(\ep)n$ eigenvalues below $\ep$ with probability
tending to $1$ as $n\to\infty$.
\end{thm}

A weaker version of Theorem \ref{thm:eigenvalues} with $C\left(\ep\right)n$
replaced by any function $\nu:\mathbb{N}\to\mathbb{N}$ with $\nu(n)=o(n)$
was proven in \cite[Theorem 1.10]{Hi.Th.22}. The proof of Theorem
\ref{thm:eigenvalues} and the prior result in \cite{Hi.Th.22} is
to use a min-max theorem with orthogonal functions that localise around
short closed geodesics each separating off two cusps from the surface.
The strengthening of the lower bound obtained here is due to the strong
Weil-Petersson volume asymptotics proven later that allow for simultaneous
control over a linear (rather than the previous sub-linear) number
of such curves that hold with high probability. 
\begin{rem}
The constant $C(\ep)$ can be given explicitly. If $\ep<\frac{1}{4}$,
the constant $C\left(\varepsilon\right)$ can be taken to be $\frac{1}{25\pi}\varepsilon J_{1}(j_{0})I_{1}\left(\frac{j_{0}\varepsilon}{12\pi}\right).$
\end{rem}

It is interesting to ask whether or not the eigenvalues guaranteed
by Theorem \ref{thm:eigenvalues} are cusp forms. We believe the answer
is no although we do not have a proof. We note that when $g=0$ or
$1$, any small eigenvalue is residual \cite{Hu83} (i.e. arises as
a pole of the scattering matrix). It is conjectured by Otal and Rosas
\cite{Ot.Ro09} that on any surface in $\mathcal{M}_{g,n}$, the number
of small cuspidal eigenvalues is bounded above by $2g-3$. It would
then follow that the number of residual eigenvalues on a random surface
in $\mathcal{M}_{g,n}$ is linear in $n$ with probability tending
to $1$ as $n\to\infty$. 

\subsection{Relative frequencies of closed curves\protect\label{subsec:Applications2}}

Our next result is on the relative frequencies of simple and non-simple
closed geodesics. For $X\in\mathcal{M}_{g,n}$, we define $N\left(X,L\right)$
to be the number of unoriented primitive closed geodesics on $X$
of length less than $L$. Similarly we define $N^{s}\left(X,L\right)$
(resp. $N^{\text{ns}}\left(X,L\right)$) to be the number of such
curves that are simple (resp. non-simple). By the results of \cite{Hi.Th.22},
a random surface with many cusps has lots of geodesics with lengths
at most $2\mathrm{arccosh}(3)$, all of which must be simple. The
prime geodesic theorem \cite{De1942} says that for any $X$,
\begin{equation}
N\left(X,L\right)\sim\frac{e^{L}}{2L}\label{eq:PGT}
\end{equation}
as $L\to\infty$, whilst by \cite{Mi2008} (see also \cite{Ri2001}),
\begin{equation}
N^{s}\left(X,L\right)\sim\eta(X)L^{6g-6+2n},\label{eq:PGT_simple}
\end{equation}
where $\eta:\mathcal{M}_{g,n}\to\mathbb{R}_{+}$ is a continuous proper
function. It then follows from (\ref{eq:PGT}) and (\ref{eq:PGT_simple})
that for any fixed $X\in\mathcal{M}_{g,n}$ one has 
\[
N^{s}\left(X,L\right)\leqslant\ep(L)N^{\mathrm{ns}}\left(X,L\right),
\]
for large enough $L$, where $\ep\left(L\right)\to0$ as $L\to\infty$.
In other words, non-simple geodesics become more abundant at large
enough length scales. It is natural to ask at what scale this transition
occurs.

This problem has recently been studied for random large-genus compact
surfaces. It was shown by Wu and Xue \cite[Theorem 4]{Wu.Xu22a} that
for a Weil-Petersson random genus $g$ compact surface, as $g\to\infty$,
most geodesics of length much shorter than $\sqrt{g}$ are simple
and non-separating whereas most geodesics of length much longer than
$\sqrt{g}$ are non-simple. This confirmed a conjecture made by Lipnowski
and Wright \cite{Li.Wr21}. Subsequently, Dozier and Sapir \cite[Theorem 1.1]{Do.Sa2023}
gave an alternative proof of the fact that most geodesics much longer
than $\sqrt{g}\log g$ are non-simple via dynamical methods which
also apply to other random models. 

We study the analogous problem for random surfaces with many cusps
and prove the following.
\begin{thm}
\label{thm:simple/non-simple-intro}Let $L=L(n)>0$ be any function
with $L\to\infty$ as $n\to\infty$ and $L=O\left(\log n\right)$.
Then there exists a function $\varepsilon(n)$ with $\varepsilon\to0$
as $n\to\infty$ such that 
\[
\mathbb{P}_{g,n}\left(X\in\mathcal{M}_{g,n}:N^{\mathrm{s}}(X,L)\leq\varepsilon(n)N^{\mathrm{ns}}(X,L)\right)\to1,
\]
as $n\to\infty$.
\end{thm}

In other words, for any $g\geqslant0$ and large $n$, on a typical
surface in $\mathcal{M}_{g,n}$, most long geodesics (which are not
too long) are non-simple. This is in stark contrast to the large genus
case. 
\begin{rem}
We believe the conclusion of Theorem \ref{thm:simple/non-simple-intro}
is true for any $L\to\infty$ without the $L=O\left(\log n\right)$
condition. The condition $L=O\left(\log n\right)$ is essentially
an artifact of our current error term $\frac{1}{n^{\frac{1}{4}}}\cosh\left(\frac{L}{2}\right)$
for Weil-Petersson volumes of moduli spaces in Theorem \ref{thm:I0-asymptotic}
which causes problems when $L$ is too large. Removing the condition
could be done by improving error terms in Theorem \ref{thm:int-asymptotic}
or possibly employing a more involved method such as \cite{Do.Sa2023,Wu.Xu2022},
which we do not pursue here.
\end{rem}

We now highlight a related problem. Let $N_{\mathrm{sep}}^{s}\left(X,L\right)$
(resp. $N_{\mathrm{nonsep}}^{s}\left(X,L\right)$) count the number
of separating (resp. non-separating) closed geodesics on $X$ of length
at most $L$. By the work of Mirzakhani \cite{Mi2008}, on any surface
$X\in\mathcal{M}_{g,n}$,
\[
\lim_{L\to\infty}\frac{N_{\mathrm{sep}}^{s}\left(X,L\right)}{N_{\mathrm{nonsep}}^{s}\left(X,L\right)}=\frac{c_{g,n,\mathrm{sep}}}{c_{g,n,\mathrm{nonsep}}},
\]
where $\frac{c_{g,n,\mathrm{sep}}}{c_{g,n,\mathrm{nonsep}}}\in\mathbb{Q}$
depends only on $g$ and $n$. The asymptotic behaviour of $\frac{c_{g,n,\mathrm{sep}}}{c_{g,n,\mathrm{nonsep}}}$
in $g$ and $n$ was recently studied by Delecroix, Goujard, Zograf
and Zorich in \cite{De.Go.Zo.Zo2023} and by Ren \cite{Re2023}. Interestingly,
the asymptotics of $\frac{c_{g,n,\mathrm{sep}}}{c_{g,n,\mathrm{nonsep}}}$
are not particularly sensitive to the ratio $\frac{n}{g}$, in contrast
to many geometric quantities \cite{Zo93,Sh.Wu22,Hi.Th.22}. 

For a random surface in $\mathcal{M}_{g,n}$ with $n$ large, there
is an abundance of very short separating simple closed geodesics \cite{Hi.Th.22}.
This indicates a transition in the ratio of $N_{\mathrm{sep}}^{s}\left(X,L\right)$
vs $N_{\mathrm{nonsep}}^{s}\left(X,L\right)$ for some $L=L(n)$.
At what scale does this transition occur for a random surface?

\subsection{Asymptotics of Weil-Petersson volumes}

We consider the Weil-Petersson volume $V_{g,n}\left(\ell_{1},\dots,\ell_{k}\right)$
of the moduli space $\mathcal{M}_{g,n}\left(\ell_{1},\dots,\ell_{k}\right)$
of bordered hyperbolic surfaces of genus $g$ with $n-k$ cusps and
$k$ labelled geodesic boundaries with lengths $\ell_{1},\dots,\ell_{k}\geqslant0$. 

After Mirzakhani's celebrated thesis works \cite{Mi2007,Mi07a}, the
asymptotic behaviour of $V_{g,n}\left(\ell_{1},\dots,\ell_{k}\right)$
is crucial for computing the integrals of geometric functions with
respect to the Weil-Petersson volume form (or from a probabilistic
perspective, computing moments of random variables with respect to
$\mathbb{P}_{\mathrm{WP}}$). A key insight in the large-genus case
is the sinh expansion, appearing first in the work of Mirzakhani and
Petri \cite{Mi.Pe19}, which states

\begin{equation}
\frac{V_{g,n}\left(\ell_{1},\dots,\ell_{k}\right)}{V_{g,n}}=\prod_{i=1}^{k}\frac{\sinh\left(\frac{\ell_{i}}{2}\right)}{\left(\frac{\ell_{i}}{2}\right)}\left(1+O_{n}\left(\frac{\sum\ell_{i}^{2}}{g}\right)\right),\label{eq:sinh-approx}
\end{equation}
as $g\to\infty$. The estimate (\ref{eq:sinh-approx}), or variants
of it (see the works \cite{Mi.Pe19,Ni.Wu.Xu2023,An.Mo2022,Wu.Xu22a}),
have been paramount in understanding questions about the spectral
gap, eigenvalue distribution and geometry of geodesics on random hyperbolic
surfaces of large genus (Section \ref{subsec:Related-work}).

The main technical contribution of this paper is to develop an analogue
of (\ref{eq:sinh-approx}) for the large-$n$ case. We prove the following.
\begin{thm}
\label{thm:I0-asymptotic}For any $g,k\geqslant0$, 
\[
\frac{V_{g,n}(\ell_{1},\dots,\ell_{k})}{V_{g,n}}=\prod_{i=1}^{k}I_{0}\left(\frac{j_{0}}{2\pi}\ell_{i}\right)+O_{g}\left(\frac{1}{n^{\frac{1}{4}}}\prod_{i=1}^{k}\cosh\left(\frac{\ell_{i}}{2}\right)\right),
\]
where $I_{0}$ is the modified Bessel function of the first kind and
$j_{0}$ is the first positive zero of the Bessel function of the
first kind $J_{0}$.
\end{thm}

By a result of Mirzakhani \cite{Mi07a} (c.f. Theorem \ref{thm:Mirz-vol-exp}),
the intersection numbers $[\tau_{0}^{n-k}\tau_{d_{1}\cdots}\tau_{d_{k}}]_{g,n}$
of tautological classes on $\overline{\mathcal{M}_{g,n}}$ (see Section
\ref{sec:Background}) play a central role in understanding $V_{g,n}\left(\ell_{1},\dots,\ell_{k}\right)$.
For the large-genus case, the formula (\ref{eq:sinh-approx}) can
be deduced from an estimate, e.g. \cite[Page 286]{Mi2013}, of the
form

\begin{equation}
\frac{[\tau_{d_{1}\cdots}\tau_{d_{n}}]_{g,n}}{V_{g,n}}=1+O_{n}\left(\frac{\left(\sum_{i=1}^{n}d_{i}\right)^{2}}{g}\right).\label{eq:intersection-asymptotic-largeg}
\end{equation}

To prove Theorem \ref{thm:I0-asymptotic}, we develop a novel asymptotic
formula for $[\tau_{0}^{n-k}\tau_{d_{1}\cdots}\tau_{d_{k}}]_{g,n}$
as $n\to\infty$.
\begin{thm}
\label{thm:int-asymptotic}Suppose that $k:\mathbb{N}\to\mathbb{N}$
is such that $k(n)\leqslant\frac{1}{8}\log(n)$. Then for any $d_{1},\ldots,d_{k}$
with $\sum_{i=1}^{k}d_{i}\leq3g+n-3$, as $n\to\infty$,
\[
\frac{[\tau_{0}^{n-k}\tau_{d_{1}}\cdots\tau_{d_{k}}]_{g,n}}{V_{g,n}}=\prod_{i=1}^{k}\left(\frac{j_{0}}{\pi}\right)^{2d_{i}}\frac{(2d_{i}+1)}{\sqrt{\pi}}\frac{\Gamma(d_{i}+\frac{1}{2})}{\Gamma(d_{i}+1)}+O_{g}\left(\frac{\sum_{i=1}^{k}d_{i}}{n^{\frac{1}{4}}}\right),
\]
where $j_{0}$ is the first positive zero of the Bessel function of
the first kind $J_{0}$.
\end{thm}

To our knowledge, Theorem \ref{thm:int-asymptotic} provides the first
result in the literature that studies the intersection numbers $[\tau_{0}^{n-k}\tau_{d_{1}}\cdots\tau_{d_{k}}]_{g,n}$
in the $n\to\infty$ regime. One difficulty with studying this quantity
is that, as opposed to the $g\to\infty$ regime, the complexity of
recursive formulae satisfied by the intersection numbers increases
rapidly in $n$. In particular, the number of terms in such formulae
that contribute to the asymptotic leading order grows rapidly. We
explain more about the widespread interest of the asymptotics of intersection
numbers in Section \ref{subsec:Related-work}. 

\subsection{Related work\protect\label{subsec:Related-work}}

We now put our work into context with existing literature.

\subsubsection*{Large $n$ regime}

The spectral geometric properties of surfaces in $\mathcal{M}_{g,n}$
in the $n\to\infty$ regime is not well studied. Early work by Zograf
\cite{Zo93} and Manin and Zograf \cite{Ma.Zo00} focussed on the
asymptotics of Weil-Petersson volumes of the moduli space for unbordered
surfaces which we replicate in Theorem \ref{thm:Manin-Zograf}. The
spectrum of the Laplacian has been studied by Zograf \cite{Zo1987}
where the deterministic bound (\ref{eq:zograf-lambda1}) on the first
non-zero eigenvalue was obtained, and we proved a sub-linear in $n$
version of Theorem \ref{thm:eigenvalues} in \cite{Hi.Th.22} for
random surfaces. 

Additionally, counting functions for the number of closed geodesics
whose lengths are in a shrinking window were shown to be Poisson distributed
in the $n\to\infty$ regime in \cite{Hi.Th.22} allowing for the expected
systole size and distribution of short geodesics to be studied on
Weil-Petersson random surfaces. The Bers' constant for deterministic
punctured spheres (genus zero, $n$ cusps) has also been studied.
The length of a pants decomposition $\mathcal{P}$ of $X\in\mathcal{M}_{0,n}$
is the length of the longest geodesic in $\mathcal{P}$. The Bers'
constant associated with $X$ is the shortest pants length among all
pants decompositions of $X$. The optimal Bers' constant $\mathcal{B}_{0,n}$
is the supremum of the Bers' constants associated to every surface
in $\mathcal{M}_{0,n}$. Balacheff and Parlier \cite{Ba.Pa2012} then
proved the following.
\begin{thm}
\label{thm:bers}If $X\in\mathcal{M}_{0,n}$ has Bers' constant equal
to $\mathcal{B}_{0,n}$, that is, any pants decomposition contains
a geodesic with length at least $\mathcal{B}_{0,n}$, then all simple
closed geodesics of $X$ have length strictly greater than $2\mathrm{arcsinh}(1)$.
\end{thm}

This contrasts the typical behaviour of Weil-Petersson random surfaces
with $n\to\infty,$where there are many curves on length scales $\mathrm{const}\cdot\frac{1}{\sqrt{n}}$.
In comparison with Theorem \ref{thm:eigenvalues} where the many short
curves give rise to linear in $n$ many exceptional eigenvalues, a
surface as in Theorem \ref{thm:bers} may be a candidate to exhibit
few exceptional eigenvalues.

Finally, we mention forthcoming work of Budd and Curien \cite{Bu.Cu2024}
where a connection between Weil-Petersson random surfaces in $\mathcal{M}_{0,n}$
and the scaling limits of random planar maps has been established.
Through this, the authors have investigated the local geometry of
a Weil-Petersson random surface in $\mathcal{M}_{0,n}$ by proving
pointed Benjamini-Schramm convergence to a random surface with genus
zero and countably many punctures. In addition, they show that globally,
after cutting the cusps at horocycles of fixed length and appropriately
rescaling the hyperbolic distance, the compact core converges as a
metric space in the Gromov-Hausdorff sense to a scaled Brownian sphere.
\begin{rem}
It is known that in many settings that Benjamini-Schramm convergence
implies the convergence of spectral measures of the Laplacian. For
example, in \cite{Ab.Be.Bi.Ge.Ni.Ra.Sa2017} it is shown that if $\Gamma_{n}$
are a uniformly discrete sequence of lattices in a connected semi-simple
Lie group $G$ with trivial centre and maximal compact subgroup $K$,
then Benjamini-Schramm convergence of the manifold $\Gamma_{n}\backslash G/K$
to $G/K$ implies convergence of the spectral measure to the normalised
Plancherel measure on $L^{2}(G)$. It would be interesting to investigate
the spectral measure for Weil-Petersson random surfaces in the large-$n$
regime and it is likely that precise asymptotics for Weil-Petersson
volumes as in Corollary \ref{thm:I0-asymptotic} will be helpful for
this.
\end{rem}

\subsubsection*{Large $g$ regime\protect\label{subsec:Large--regime}}

The literature for studying random surfaces in the $g\to\infty$ is
much more well established whereby the sinh approximation of (\ref{eq:sinh-approx})
and its variants have played an important role in obtaining the current
state of the art results.

The first application of (\ref{eq:sinh-approx}) was in \cite{Mi.Pe19},
where it was applied to prove Poisson statistics for the number of
closed geodesics whose lengths are in a fixed window on closed surfaces.
It was also used in investigating the size of the first non-zero eigenvalue
(spectral gap) in \cite{Wu.Xu2022,Li.Wr21} where the authors independently
prove that a Weil-Petersson random closed surface of large genus has
spectral gap at least $\frac{3}{16}-\ep$. Recently, this result was
improved to $\frac{2}{9}-\ep$ by Anantharaman and Monk \cite{An.Mo2023},
where they applied the sharpest known version of (\ref{eq:sinh-approx}),
proven in \cite{An.Mo2022}. See also the works \cite{Ma.Na.Pu2022,Ma.Na2020,Hi.Ma22}
for spectral gap results for random covers. Other applications have
been in proving delocalisation estimates on Laplacian eigenfunctions
\cite{Le.Sa17,Gi.Le.Sa.Th21,Th2020}, connecting eigenvalue distribution
to random matrix theory \cite{Ru22}, and investigating the Cheeger
constant and behaviour of the systole \cite{Pa.Wu.Xu21,Ni.Wu.Xu2023}
for closed surfaces. The sinh-approximation has also been used to
prove results of a similar flavour when $n$ grows with the genus
$g$, where the behaviour of geodesics and eigenvalues changes dependent
on the $n$ and $g$ regime \cite{Hi.22,Sh.Wu22}.

Theorem \ref{thm:I0-asymptotic}, the large-$n$ analogue of (\ref{eq:sinh-approx}),
offers an explanation for the contrasting results seen for the spectral
geometry in the large $g$ and $n$ regimes: the leading order asymptotics
in $n$ are subdominant to those in $g$ leading to shorter scale
geometry to permeate. As with (\ref{eq:sinh-approx}), we expect that
our result will be useful for investigating other spectral geometric
questions for Weil-Petersson random surfaces in addition to the results
that we prove here. These applications are of significant interest
in there own right and because of their sharp contrast to the behaviour
witnessed on large genus closed surfaces.

\subsubsection*{Asymptotics of intersection numbers}

Aside from appearing in Mirzakhani's formula for moduli space volumes
(c.f. Theorem \ref{thm:Mirz-vol-exp}) the intersection numbers $\left[\tau_{d_{1}}\cdots\tau_{d_{n}}\right]_{g,n}$
(after a suitable normalisation -- compare (\ref{eq:def-int-numbers})
to the definition in the introduction of \cite{Ag2021} for example),
play the role of correlation functions in Witten's model of two-dimensional
quantum gravity \cite{Wi1991}. In proving a conjecture of Witten
regarding the nature of the intersection numbers, Kontesevich \cite{Ko1992}
related them to the combinatorics of trivalent ribbon graphs with
genus $g$ and $n$ boundaries. Additionally, the intersection numbers
appear in formulae for frequencies of geodesic multicurves on hyperbolic
and random flat surfaces by Mirzakhani \cite{Mi2008} and Delecroix,
Goujard, Zograf and Zorich \cite{De.Go.Zo.Zo2021,De.Go.Zo.Zo2022,De.Go.Zo.Zo2023}
respectively. 

Computation of the intersection numbers can in principle be carried
out exactly using recursive relations of Witten and Kontesevich \cite{Wi1991,Ko1992}
and their connection to Virasoro constraints or the recursive relations
of Liu and Xu \cite{Li.Xu2009} (some of which we reproduce in Section
\ref{subsec:Intersection-numbers}). There are however, no explicit
closed form expressions of the intersection numbers outside of some
exceptional cases, and so it is of immense interest to obtain asymptotic
formulae for them instead in various regimes of $g$ and $n$. For
$g\to\infty$ (and with control over the growth of $n$ in terms of
$g$), there has been significant progress on these asymptotics made
by Mirzakhani and Zograf \cite{Mi.Zo2015} and Liu and Xu \cite{Li.Xu2014}.
When $\sum_{i=1}^{n}d_{i}=3g+n-3$, the exact form of the intersection
numbers in the $g\to\infty$ limit were conjectured by Delecroix,
Goujard, Zograf and Zorich \cite[Conjecture E.6]{De.Go.Zo.Zo2021}
and recently proven by Aggarwal \cite{Ag2021} (see also a special
case by Liu and Xu \cite{Li.Xu2014}). Our result, Theorem \ref{thm:int-asymptotic},
marks significant progress in understanding the asymptotics of the
intersection numbers $[\tau_{0}^{n-k}\tau_{d_{1}}\cdots\tau_{d_{k}}]_{g,n}$
in the $n\to\infty$ regime.

\subsection{Overview of the proofs of Theorems \ref{thm:I0-asymptotic} and \ref{thm:int-asymptotic}}

\label{subsec:proof-overviews}

We first outline the proof of Theorem \ref{thm:int-asymptotic} from
which Theorem \ref{thm:I0-asymptotic} follows. We use the following
recursive formula, Lemma \ref{lem:Mirzakhani-Zograf-(Ib)}, which
follows from \cite[Propositions 3.3 and 3.4]{Li.Xu2009}, c.f. \cite[eq. (Ib)]{Mi.Zo2015}.
\begin{align}
[\tau_{0}^{2}\tau_{\ell+1}\prod_{i=1}^{n}\tau_{d_{i}}]_{g,n+3}= & [\tau_{0}^{4}\tau_{\ell}\prod_{i=1}^{n}\tau_{d_{i}}]_{g-1,n+5}+8\sum_{\substack{I\sqcup J=\{1,...,n\}\\
g_{1}+g_{2}=g
}
}[\tau_{0}^{2}\tau_{\ell}\prod_{i\in I}\tau_{d_{i}}]_{g_{1},|I|+3}[\tau_{0}^{2}\prod_{i\in J}\tau_{d_{i}}]_{g_{2},|J|+2}\nonumber \\
 & +4\sum_{\substack{I\sqcup J=\{1,...,n\}\\
g_{1}+g_{2}=g
}
}[\tau_{0}\tau_{\ell}\prod_{i\in I}\tau_{d_{i}}]_{g_{1},|I|+2}[\tau_{0}^{3}\prod_{i\in J}\tau_{d_{i}}]_{g_{2},|J|+3}.\label{eq:Recurs-intro}
\end{align}
We explain the method for the simpler case where $d_{i}=0$ for $i=1,\dots,n$.
We remind the reader that $\left[\tau_{0}^{n}\right]_{g,n}=V_{g,n}$.
The starting point for our analysis is the large-$n$ asymptotic for
$V_{g,n}$ proved by Manin and Zograf, c.f. Theorem \ref{thm:Manin-Zograf},
which states that for $g\geqslant0$ fixed, 
\begin{equation}
V_{g,n}=\frac{\left(2\pi^{2}\right)^{3g+n-3}}{x_{0}^{n}}n!(n+1)^{\frac{5g-7}{2}}\left(B_{g}+O_{g}\left(\frac{1}{n}\right)\right),\label{eq:manin-zograf-intro}
\end{equation}
as $n\to\infty$, for some constant $B_{g}>0$ with $x_{0}=-\frac{1}{2}j_{0}J_{0}^{'}(j_{0})$.
Using (\ref{eq:manin-zograf-intro}), we study (\ref{eq:Recurs-intro})
inductively. By comparing each term of (\ref{eq:Recurs-intro}) with
(\ref{eq:manin-zograf-intro}), we are able to prove (e.g. Lemma \ref{lem:rec_rel})
that
\[
\frac{[\tau_{\ell+1}\tau_{0}^{n+2}]_{g,n+3}}{V_{g,n+3}}=C_{\ell+1}+O_{g}\left(\frac{l}{n^{\frac{1}{4}}}\right),
\]
where 
\begin{equation}
C_{\ell+1}\eqdf8C_{\ell}\sum_{i=0}^{\infty}\frac{V_{0,i+3}}{(2\pi^{2})^{i+1}(i+1)!}x_{0}^{i+1}+4\sum_{i=0}^{\infty}\frac{[\tau_{0}^{i+2}\tau_{\ell}]_{0,i+3}}{(2\pi^{2})^{i+1}(i+1)!}x_{0}^{i+1}.\label{eq:recurs-consts-intro}
\end{equation}
It is an important point that the constants $C_{\ell}$ involve only
genus-$0$ intersection numbers, regardless of the value of $g$.
Our goal is now to compute the constants $C_{\ell}$ explicitly. To
do this we need to compute the constant
\[
b_{\ell}\eqdf\sum_{i=0}^{\infty}\frac{[\tau_{0}^{i+2}\tau_{\ell}]_{0,i+3}}{(2\pi^{2})^{i+1}(i+1)!}x_{0}^{i+1},
\]
for each $\ell$ which then allows us to solve the recurrence relation
(\ref{eq:recurs-consts-intro}) to compute $C_{\ell}$. We consider
the generating functions
\[
\Phi_{\ell}\left(x\right)\eqdf\sum_{i=3}^{\infty}\frac{[\tau_{0}^{i-1}\tau_{\ell}]_{0,i}}{(2\pi^{2})^{i-2}i!}x^{i},
\]
noting that $b_{\ell}=\Phi_{\ell}^{''}\left(x_{0}\right)$. Crucially,
for the case that $\ell=0$, it was shown in \cite[eq. (0.8)]{Ka.Ma.Za96}
that $y(x)\eqdf\Phi_{0}^{''}\left(x\right)$ can be obtained by inverting
the function $x(y)\eqdf-\sqrt{y}J_{0}^{'}\left(2\sqrt{y}\right)$
which gives that $b_{0}=\frac{j_{0}^{2}}{4}$. This follows from a
recursive formula for genus-$0$ Weil-Petersson volumes due to Zograf
\cite{Zo93}. We observe, in Lemma \ref{lem:phi_k-rec}, that (\ref{eq:Recurs-intro})
implies that $\Phi_{\ell}\left(x\right)$ satisfies the following
family of ODE's
\begin{equation}
\Phi_{\ell}^{'''}(x)=8\Phi_{0}^{''}(x)\Phi_{\ell-1}^{'''}(x)+4\Phi_{0}^{'''}(x)\Phi_{\ell-1}^{''}(x).\label{eq:ODE's-intro}
\end{equation}
We solve (\ref{eq:ODE's-intro}) for $\Phi_{\ell}^{''}\left(x\right)$
to find that 
\[
\Phi_{\ell}^{''}(x)=\frac{2^{3\ell+1}}{\sqrt{\pi}}\frac{\Gamma(\ell+\frac{3}{2})}{\Gamma(\ell+2)}\Phi_{0}^{''}(x)^{\ell+1}.
\]
Using $b_{0}=\frac{j_{0}^{2}}{4}$, we learn that 
\[
b_{\ell}=\frac{2^{3\ell+1}}{\sqrt{\pi}}\frac{\Gamma(\ell+\frac{3}{2})}{\Gamma(\ell+2)}\left(\frac{j_{0}^{2}}{4}\right)^{\ell+1}.
\]
Solving (\ref{eq:recurs-consts-intro}) then gives 
\[
C_{\ell}=\left(\frac{j_{0}}{\pi}\right)^{2\ell}\frac{(2\ell+1)}{\sqrt{\pi}}\frac{\Gamma(\ell+\frac{1}{2})}{\Gamma(\ell+1)}.
\]
After some more technical computations, the general case of $[\tau_{0}^{2}\tau_{\ell+1}\prod_{i=1}^{n-3}\tau_{d_{i}}]_{g,n}$
for $d_{1},\dots,d_{k}>0$ essentially reduces to the above considerations.

Finally, since by \cite[Theorem 1.1]{Mi07a}, 
\[
\text{\ensuremath{V_{g,n}\left(2x_{1},\dots,2x_{n}\right)}}=\sum_{\substack{d_{1},\dots,d_{n}\\
|d|\leqslant3g+n-3
}
}\left[\prod_{i=1}^{n}\tau_{d_{i}}\right]_{g,n}\frac{x_{1}^{2d_{1}}}{\left(2d_{1}+1\right)!}\cdot\dots\cdot\frac{x_{n}^{2d_{n}}}{\left(2d_{n}+1\right)!},
\]
Theorem \ref{thm:I0-asymptotic} follows from Theorem \ref{thm:int-asymptotic}
by recognising the Taylor expansion for $\prod_{i=1}^{k}I_{0}(x_{i})$,
where we recall 
\[
I_{0}(x)=\sum_{d=0}^{\infty}\left(\frac{x}{2}\right)^{2d}\frac{1}{(d!)^{2}}.
\]

\subsection{Outline of the paper}

In Section \ref{sec:Background} we will gather together the definitions
of the objects we are interested in studying along with some results
from the literature that we shall use throughout the article. Section
\ref{sec:Preliminary-Volume-Computations} proves some preliminary
results involving the moduli space volumes that we will frequently
apply when computing the asymptotics of the intersection numbers.
Mostly these results are technical and their proofs can be skipped
over without loss for reading the remainder of the paper. In Section
\ref{sec:Asymptotics-of-Intersection} we prove Theorems \ref{thm:int-asymptotic}
and \ref{thm:I0-asymptotic} starting with some special cases that
lead into an inductive argument later in the section. The remaining
two sections are devoted to proving the respective applications outlined
previously.

\subsection{Notation}

For real valued functions $f,h$ depending on a parameter $n$ we
write $f\ll h$ or $f=O\left(h\right)$ if there exists $C,N>0$ such
that $|f(n)|\leqslant Ch(n)$ for all $g>N$. We add subscripts to
the $\ll$ sign if the constant $C,N$ depend on another variable.
E.g. we write $f\ll_{\epsilon}h$ if there exists $C=C(\ep),N=N(\ep)$
such that $|f(n)|\leqslant Ch(n)\text{ for all }g>N$. We write $f\sim h$
if $f\ll h$ and $h\ll f$.

\section{Background\protect\label{sec:Background}}

In this section we briefly introduce the necessary background and
state some results we will use later in the paper. 

\subsection{Moduli space }

Let $\Sigma_{g,n,b}$ denote a topological surface with genus $g$,
$n$ labelled punctures and $b$ labelled boundary components where
$2g+n+b\geqslant3$. A marked surface of signature $\left(g,n,b\right)$
is a pair $\left(X,\varphi\right)$ where $X$ is a hyperbolic surface
and $\varphi:\Sigma_{g,n,b}\to X$ is a homeomorphism. Given $\left(\ell_{1},...,\ell_{b}\right)\in\mathbb{R}_{>0}^{b}$,
we define the Teichmüller space $\mathcal{T}_{g,n,b}\left(\ell_{1},\dots,\ell_{b}\right)$
by
\begin{align*}
\mathcal{T}_{g,n,b}\left(\ell_{1},...,\ell_{b}\right) & \stackrel{\text{def}}{=}\left\{ \substack{\text{\text{Marked surfaces} }\left(X,\varphi\right)\text{ of signature }\left(g,n,b\right)\\
\text{with labelled totally geodesic boundary components }\\
\left(\beta_{1},\dots,\beta_{b}\right)\text{ with lengths }\left(\ell_{1},\dots,\ell_{b}\right)
}
\right\} /\sim,
\end{align*}
where $\left(X_{1},\varphi_{1}\right)\sim\left(X_{2},\varphi_{2}\right)$
if and only if there exists an isometry $m:X_{1}\to X_{2}$ such that
$\varphi_{2}$ and $m\circ\varphi_{1}$ are isotopic. Let $\text{Homeo}^{+}\left(\Sigma_{g,n,b}\right)$
denote the group of orientation preserving homeomorphisms of $\Sigma_{g,c,d}$
which leave every boundary component set-wise fixed and do not permute
the punctures. Let $\text{Homeo}_{0}^{+}\left(\Sigma_{g,n,b}\right)$
denote the subgroup of homeomorphisms isotopic to the identity. The
mapping class group is defined as 
\[
\text{MCG}_{g,n,b}\stackrel{\text{def}}{=}\text{Homeo}^{+}\left(\Sigma_{g,n,b}\right)/\text{Homeo}_{0}^{+}\left(\Sigma_{g,n,b}\right).
\]
$\text{Homeo}^{+}\left(\Sigma_{g,n,b}\right)$ acts on $\mathcal{T}_{g,n,b}\left(\ell_{1},...,\ell_{b}\right)$
by pre-composition of the marking and we define the moduli space $\mathcal{M}_{g,n,b}\left(\ell_{1},...,\ell_{b}\right)$
by 
\[
\mathcal{M}_{g,n,b}\left(\ell_{1},...,\ell_{b}\right)\stackrel{\text{def}}{=}\mathcal{T}_{g,n,b}\left(\ell_{1},...,\ell_{b}\right)/\text{MCG}_{g,n,b}.
\]
By convention, a geodesic of length $0$ is a cusp and we suppress
the distinction between cusps and boundary components in our notation
by allowing $\ell_{i}\geqslant0$. In particular we write

\[
\mathcal{M}_{g,n+b}=\mathcal{M}_{g,n,b}\left(0,\dots,0\right).
\]

\subsection{\protect\label{subsec:Intersection-numbers}Intersection numbers
and Weil-Petersson volumes}

Let $\omega_{WP}$ denote the Weil-Petersson symplectic form on $\mathcal{T}_{g,n}\left(\ell_{1},\dots,\ell_{n}\right)$.
$\omega_{WP}$ is invariant under the action of the mapping class
group \cite{Go1984} and descends to a symplectic form on $\mathcal{M}_{g,n}\left(\ell_{1},\dots,\ell_{n}\right)$
where it induces the volume form
\[
\text{dVol}_{WP}\stackrel{\text{def}}{=}\frac{1}{\left(3g-3+n\right)!}\bigwedge_{i=1}^{3g-3+n}\omega_{WP}.
\]
We write $V_{g,n}\left(\ell_{1},\dots,\ell_{n}\right)$ to denote
$\text{Vol}_{WP}\left(\mathcal{M}_{g,n}\left(\ell_{1},\dots,\ell_{n}\right)\right)$,
which is finite. We define the Weil-Petersson probability measure
$\mathbb{P}_{g,n}$ on $\mathcal{M}_{g,n}$ by normalising $\text{dVol}_{WP}$.

Let $\overline{\mathcal{M}}_{g,n}$ be the Deligne-Mumford compactification
of $\mathcal{M}_{g,n}$. There are $n$ tautological line bundles
$\mathcal{L}_{i}$ over $\overline{\mathcal{M}}_{g,n}$ whose fiber
at $X\in\overline{\mathcal{M}}_{g,n}$ is the cotangent space at the
$i$th marked point on $X$. We define the $\psi$-classes $\psi_{i}\eqdf c_{1}\left(\mathcal{L}_{i}\right)$
where $c_{1}$ denotes the first Chern class of the bundle $\mathcal{L}_{i}$.
For $d=\left(d_{1},\dots,d_{n}\right)\in\mathbb{Z}_{\geqslant0}^{n}$
with $\left|d\right|\eqdf\sum_{i=1}^{n}d_{i}\leqslant3g+n-3$, we
define 
\begin{equation}
\left[\tau_{d_{1}}\dots\tau_{d_{n}}\right]_{g,n}\eqdf\frac{2^{2\left|d\right|}\prod_{i=1}^{n}\left(2d_{i}+1\right)!!}{(3g+n-3-|d|)!}\int_{\overline{\mathcal{M}}_{g,n}}\psi_{1}^{d_{1}}\cdots\psi_{n}^{d_{n}}\omega_{WP}^{3g+n-3-\left|d\right|}.\label{eq:def-int-numbers}
\end{equation}
If $|d|>3g+n-3$ , $\left[\tau_{d_{1}}\dots\tau_{d_{n}}\right]_{g,n}$
is taken to be identically $0$. For further background, see \cite{Ar.Co1996}.
Mirzakhani proved that $V_{g,n}\left(\ell_{1},\dots,\ell_{n}\right)$
is a polynomial in $\ell_{1},\dots,\ell_{n}$ with coefficients given
by the intersection numbers $\left[\tau_{d_{1}}\dots\tau_{d_{n}}\right]_{g,n}$.
\begin{thm}[{\cite[Theorem 1.1]{Mi07a}}]
\label{thm:Mirz-vol-exp}For $n>0$ and $\ell_{1},\dots,\ell_{n}>0$,
\[
\text{\ensuremath{V_{g,n}\left(2\ell_{1},\dots,2\ell_{n}\right)}}=\sum_{\substack{d_{1},\dots,d_{n}\\
|d|\leqslant3g+n-3
}
}\left[\prod_{i=1}^{n}\tau_{d_{i}}\right]_{g,n}\frac{\ell_{1}^{2d_{1}}}{\left(2d_{1}+1\right)!}\cdot\dots\cdot\frac{\ell_{n}^{2d_{n}}}{\left(2d_{n}+1\right)!}.
\]
\end{thm}

By Theorem \ref{thm:Mirz-vol-exp}, we can study the asymptotics of
$\frac{V_{g,n}\left(\ell_{1},\dots,\ell_{n}\right)}{V_{g,n}}$ by
understanding the asymptotics of $\frac{\left[\prod_{i=1}^{n}\tau_{d_{i}}\right]_{g,n}}{V_{g,n}}$.
To this end, we rely heavily on the following recursive formula which
follows from \cite[Propositions 3.3 and 3.4]{Li.Xu2009}, c.f. \cite[eq. (Ib)]{Mi.Zo2015}.
\begin{lem}[\cite{Li.Xu2009}]
\label{lem:Mirzakhani-Zograf-(Ib)}
\begin{align*}
[\tau_{0}^{2}\tau_{\ell+1}\prod_{i=1}^{n}\tau_{d_{i}}]_{g,n+3} & =[\tau_{0}^{4}\tau_{\ell}\prod_{i=1}^{n}\tau_{d_{i}}]_{g-1,n+5}+8\sum_{\substack{I\sqcup J=\{1,...,n\}\\
g_{1}+g_{2}=g
}
}[\tau_{0}^{2}\tau_{\ell}\prod_{i\in I}\tau_{d_{i}}]_{g_{1},|I|+3}[\tau_{0}^{2}\prod_{i\in J}\tau_{d_{i}}]_{g_{2},|J|+2}\\
 & +4\sum_{\substack{I\sqcup J=\{1,...,n\}\\
g_{1}+g_{2}=g
}
}[\tau_{0}\tau_{\ell}\prod_{i\in I}\tau_{d_{i}}]_{g_{1},|I|+2}[\tau_{0}^{3}\prod_{i\in J}\tau_{d_{i}}]_{g_{2},|J|+3}.
\end{align*}
\end{lem}

We also record a basic estimate for $\left[\tau_{d_{1}},\dots,\tau_{d_{n}}\right]_{g,n}$
in terms of Weil-Petersson volumes.
\begin{lem}[{\cite[Lemma 3.2, Part (1)]{Mi2013}}]
\label{lem:intersection-triv-bound}For any $d_{1},\dots d_{n}$,
\[
\left[\tau_{d_{1}},\dots,\tau_{d_{n}}\right]_{g,n}\leqslant V_{g,n}.
\]
\end{lem}

We shall make essential use of the following large-$n$ asymptotic
formula for $V_{g,n}$ due to Manin and Zograf \cite{Ma.Zo00}.
\begin{thm}[{\cite[Theorem 6.1]{Ma.Zo00}}]
\label{thm:Manin-Zograf} For any $g\geqslant0$, there exists a
constant $B_{g}>0$ such that
\[
V_{g,n}=\frac{\left(2\pi^{2}\right)^{3g+n-3}}{x_{0}^{n}}n!(n+1)^{\frac{5g-7}{2}}\left(B_{g}+O_{g}\left(\frac{1}{n}\right)\right),
\]
where $x_{0}\eqdf-\frac{1}{2}j_{0}J_{0}^{'}(j_{0})$ and $J_{0}$
is the Bessel function of the first kind with $j_{0}$ its first positive
zero.
\end{thm}

The proof of Theorem \ref{thm:int-asymptotic} relies on analysing
the recursion of Lemma \ref{lem:Mirzakhani-Zograf-(Ib)} together
with Theorem \ref{thm:Manin-Zograf}. 

Finally we will often apply the following trivial upper bound for
$V_{g,n}\left(\ell_{1},\dots,\ell_{n}\right)$.
\begin{lem}
\label{lem:sinh-upper-bound}For any $g,n\geqslant0$ with $2g+n>2$
and $\ell_{1},\dots,\ell_{n}\geqslant0$ 
\[
\frac{V_{g,n}\left(\ell_{1},\dots,\ell_{n}\right)}{V_{g,n}}\leqslant\prod_{i=1}^{n}\frac{\sinh\left(\frac{\ell_{i}}{2}\right)}{\left(\frac{\ell_{i}}{2}\right)}.
\]
\end{lem}

This follows from Theorem \ref{thm:Mirz-vol-exp} and Lemma \ref{lem:intersection-triv-bound}
by the Taylor expansion of $\prod_{i=1}^{n}\frac{\sinh\left(\frac{\ell_{i}}{2}\right)}{\left(\frac{\ell_{i}}{2}\right)}$,
e.g. \cite[Proposition 3.1]{Mi.Pe19}.

\subsection{Mirzakhani's integration formula}

Finally, we give a brief account of Mirzakhani's integration formula
which will used in our applications.

We define a $k$-multicurve to be an ordered $k$-tuple $\Gamma=\left(\gamma_{1},...,\gamma_{k}\right)$
of disjoint non-homotopic non-peripheral simple closed curves on $\Sigma_{g,n}$,
and we write $\left[\Gamma\right]$$=\left[\gamma_{1},...,\gamma_{k}\right]$
to denote its homotopy class. The mapping class group $\text{MCG}_{g,n}$
acts on homotopy classes of multicurves and we denote the orbit containing
$\left[\Gamma\right]$ by 
\[
\mathcal{O}_{\Gamma}=\left\{ \left(\sigma\cdot\gamma_{1},...,\sigma\cdot\gamma_{k}\right)\mid\sigma\in\text{MCG}_{g,n}\right\} .
\]
Given a simple, non-peripheral closed curve $\gamma$ on $\Sigma_{g,n}$,
for $\left(X,\varphi\right)\in\mathcal{T}_{g,n}$ we define $\ell_{\gamma}\left(X\right)$
to be the length of the unique geodesic in the free homotopy class
of $\varphi\left(\gamma\right)$. Then given a function $f:\mathbb{R}_{\geqslant0}^{k}\to\mathbb{R}_{\geqslant0}$,
for $X\in\mathcal{M}_{g,n}$ we define 
\[
f^{\Gamma}\left(X\right)\eqdf\sum_{\left(\alpha_{1},...,\alpha_{k}\right)\in\mathcal{O}_{\Gamma}}f\left(\ell_{\alpha_{1}}\left(X\right),...,\ell_{\alpha_{k}}\left(X\right)\right).
\]
Let $\Sigma_{g,n}\left(\Gamma\right)$ denote the result of cutting
the surface $\Sigma_{g,n}$ along $\left(\gamma_{1},...,\gamma_{k}\right)$,
then $\Sigma_{g,n}\left(\Gamma\right)=\sqcup_{i=1}^{s}\Sigma_{g_{i},c_{i},d_{i}}$
for some $\left\{ \left(c_{i},d_{i}\right)\right\} _{i=1}^{s}$. Each
$\gamma_{i}$ gives rise to two boundary components $\gamma_{i}^{1}$
and $\gamma_{i}^{2}$ of $\Sigma_{g,n}\left(\Gamma\right)$. Given
$\boldsymbol{x}=\left(x_{1},...,x_{k}\right)$, let $\boldsymbol{x}^{(i)}$
denote the tuple of coordinates $x_{j}$ of $\boldsymbol{x}$ such
that $\gamma_{j}$ is a boundary component of $\Sigma_{g,c_{i},d_{i}}$.
We define 
\begin{align*}
V_{g,n}\left(\Gamma,\mathbf{x}\right)\stackrel{\text{def}}{=} & \prod_{i=1}^{s}V_{g_{i},c_{i}+d_{i}}\left(\boldsymbol{x}^{(i)}\right).
\end{align*}
We can now state Mirzakhani's integration formula.
\begin{thm}[{Mirzakhani's Integration Formula \cite[Theorem 7.1]{Mi2007}}]
\label{thm:MIF}Given a $k$-multicurve $\Gamma=\left(\gamma_{1},\dots,\gamma_{k}\right)$,
\[
\int_{\mathcal{M}_{g,n}}f^{\Gamma}\left(X\right)dX=C\left(\Gamma\right)\int_{\mathbb{R}_{\geqslant0}^{k}}f\left(x_{1},...,x_{k}\right)V_{g,n}\left(\Gamma,\mathbf{x}\right)x_{1}\cdots x_{k}\mathrm{d}x_{1}\cdots\mathrm{d}x_{k},
\]
where $C\left(\Gamma\right)\leqslant1$ is an explicit constant.
\end{thm}

One can see \cite[Footnote (2)]{Wr2020} for a detailed description
of the constant $C\left(\Gamma\right)$. We note that if each component
of the multicurve $\Gamma$ is separating then $C\left(\Gamma\right)=1$.

\section{Preliminary Volume Computations\protect\label{sec:Preliminary-Volume-Computations}}

In this section, we isolate some technical computations that we will
frequently apply throughout the remainder of the paper. First, we
begin with noting a computation due to Manin and Zograf as well as
estimating the order of the tail of a related series.
\begin{lem}
\label{lem:Manin-Zograf-series}The series 
\begin{equation}
\sum_{i=0}^{\infty}\frac{V_{0,i+3}x_{0}^{i+1}}{(2\pi^{2})^{i}(i+1)!}=\frac{j_{0}^{2}}{4},\label{eq:j_0-sum}
\end{equation}
where $j_{0}$ is the first positive zero of the Bessel function $J_{0}$.
Moreover, for any fixed $g\geq0$, as $n\to\infty$

\[
\sum_{i=\lfloor\sqrt{n}\rfloor}^{\infty}\frac{V_{g,i+3}x_{0}^{i+1}}{(i+1)!(2\pi^{2})^{i}(i+4)^{\frac{5g}{2}}}\ll_{g}\frac{1}{n^{\frac{1}{4}}},
\]
and 

\[
\sum_{i=0}^{\infty}\frac{V_{g,i+3}x_{0}^{i+1}}{(i+1)!(2\pi^{2})^{i}(i+4)^{\frac{5g}{2}}}\ll_{g}1.
\]
\end{lem}

\begin{proof}
Defining 
\[
\Phi_{0}\left(x\right)\eqdf\sum_{i=0}^{\infty}\frac{V_{0,i+3}x_{0}^{i+1}}{(2\pi^{2})^{i}(i+1)!},
\]
as discussed in Section \ref{subsec:proof-overviews}, it was shown
in \cite[eq. (0.8)]{Ka.Ma.Za96} that $y(x)\eqdf\Phi_{0}''\left(x\right)$
can be obtained by inverting the Bessel function $x(y)\eqdf-\sqrt{y}J_{0}'\left(2\sqrt{y}\right)$.
The first part then follows from evaluating $y\left(x_{0}\right)=\frac{j_{0}^{2}}{4}$.
For the second part, by Theorem \ref{thm:Manin-Zograf}, as $i\to\infty$,
\[
\frac{V_{g,i+3}x_{0}^{i+1}}{(2\pi^{2})^{i}(i+1)!(i+4)^{\frac{5g}{2}}}=(2\pi^{2})^{3g}x_{0}^{-2}\frac{(i+3)(i+2)}{(i+4)^{\frac{7}{2}}}\left(B_{g}+O_{g}\left(\frac{1}{i}\right)\right)\ll_{g}\frac{1}{(i+4)^{\frac{3}{2}}},
\]
and so, as $n\to\infty$ and $i\geq\lfloor\sqrt{n}\rfloor$, we have
\[
\sum_{i=\lfloor\sqrt{n}\rfloor}^{\infty}\frac{V_{g,i+3}x_{0}^{i+1}}{(2\pi^{2})^{i}(i+1)!(i+4)^{\frac{5g}{2}}}\ll_{g}\sum_{i=\lfloor\sqrt{n}\rfloor}^{\infty}\frac{1}{(i+4)^{\frac{3}{2}}}\ll_{g}\int_{\sqrt{n}-1}^{\infty}\frac{1}{(x+4)^{\frac{3}{2}}}\mathrm{d}x\ll_{g}\frac{1}{n^{\frac{1}{4}}}.
\]
The final part follows from the limit comparison test by comparing
the summand $\frac{V_{g,i+3}x_{0}^{i+1}}{(i+1)!(2\pi^{2})^{i}(i+4)^{\frac{5g}{2}}}$
with the sequence $\frac{1}{(i+4)^{\frac{3}{2}}}$. 
\end{proof}
\begin{lem}
\label{lem:bottom-terms}For any $0\leqslant g_{1}\leqslant g$ and
$\varepsilon>0$, as $n\to\infty$,
\[
\sum_{i=0}^{\lfloor\frac{n-4}{2}\rfloor}{n-3 \choose i}\frac{V_{g_{1},i+3}V_{g-g_{1},n-i-1}}{V_{g,n}}\ll_{g}\frac{1}{n^{\frac{1}{2}-\varepsilon}}.
\]
\end{lem}

\begin{proof}
We split the summation and consider first when $0\leqslant i\leqslant\lfloor\sqrt{n}\rfloor$.
By Theorem \ref{thm:Manin-Zograf},
\begin{align*}
 & {n-3 \choose i}\frac{V_{g-g_{1},n-i-1}}{V_{g,n}}\\
= & \frac{(n-3)!}{i!(n-3-i)!}\frac{(n-i-1)!}{n!}\left(\frac{n-i}{n+1}\right)^{\frac{5(g-g_{1})-7}{2}}x_{0}^{i+1}\frac{1}{(2\pi^{2})^{3g_{1}+i+1}(n+1)^{\frac{5g_{1}}{2}}}\left(1+O_{g}\left(\frac{1}{n}\right)\right)\\
= & \frac{i+1}{n}\frac{n-1-i}{n-1}\frac{n-2-i}{n-2}\frac{x_{0}^{i+1}}{(i+1)!}\frac{1}{(2\pi^{2})^{3g_{1}+i+1}(n+1)^{\frac{5g_{1}}{2}}}\left(1+O_{g}\left(\frac{1}{\sqrt{n}}\right)\right)\\
\ll_{g} & \frac{1}{\sqrt{n}}\frac{x_{0}^{i+1}}{(i+1)!}\frac{1}{(2\pi^{2})^{3g_{1}+i+1}}\frac{1}{(n+1)^{\frac{5g_{1}}{4}}}\frac{1}{(i+4)^{\frac{5g_{1}}{2}}}.
\end{align*}
And so, using Lemma \ref{lem:Manin-Zograf-series}
\begin{align*}
\sum_{i=0}^{\lfloor\sqrt{n}\rfloor}{n-3 \choose i}\frac{V_{g_{1},i+3}V_{g-g_{1},n-i-1}}{V_{g,n}} & \ll_{g}\frac{1}{n^{\frac{1}{2}+\frac{5g_{1}}{4}}}\sum_{i=0}^{\lfloor\sqrt{n}\rfloor}\frac{V_{g_{1},i+3}x_{0}^{i+1}}{(2\pi^{2})^{3g_{1}+i+1}(i+1)!(i+4)^{\frac{5g_{1}}{2}}}\\
 & \ll_{g}\frac{1}{n^{\frac{1}{2}+\frac{5g_{1}}{4}}}.
\end{align*}
For $\lfloor\sqrt{n}\rfloor\leqslant i\leqslant\lfloor\frac{n-4}{2}\rfloor$,
we again use Theorem \ref{thm:Manin-Zograf} to obtain
\begin{align*}
 & {n-3 \choose i}\frac{V_{g-g_{1},n-i-1}}{V_{g,n}}\\
= & \frac{(n-3)!}{i!(n-3-i)!}\frac{(n-i-1)!}{n!}\left(\frac{n-i}{n+1}\right)^{\frac{5(g-g_{1})-7}{2}}x_{0}^{i+1}\frac{1}{(2\pi^{2})^{3g_{1}+i+1}(n+1)^{\frac{5g_{1}}{2}}}\left(1+O_{g}\left(\frac{1}{n}\right)\right)\\
\leqslant & 2^{\frac{7}{2}}\frac{1}{n}\frac{1}{i!}\frac{x_{0}^{i+1}}{(2\pi^{2})^{3g_{1}+i+1}(n+1)^{\frac{5g_{1}}{2}}}\left(1+O_{g}\left(\frac{1}{n}\right)\right),
\end{align*}
where we use 
\[
\left(\frac{n-i}{n+1}\right)^{\frac{5(g-g_{1})-7}{2}}\leqslant2^{\frac{7}{2}},
\]
since $i+1\leqslant\frac{n}{2}-1$ and $n-i\geqslant\frac{n}{2}+2$.
Moreover, for any $\varepsilon>0,$
\[
\frac{1}{n}\ll\frac{1}{n^{\frac{1}{2}-\varepsilon}}\frac{1}{(i+4)^{\frac{1}{2}+\varepsilon}}
\]
and so
\begin{align*}
\sum_{i=\lfloor\sqrt{n}\rfloor+1}^{\lfloor\frac{n-4}{2}\rfloor}{n-3 \choose i}\frac{V_{g_{1},i+3}V_{g-g_{1},n-i-1}}{V_{g,n}} & \ll_{g}\frac{1}{n^{\frac{1}{2}-\varepsilon}}\sum_{i=\lfloor\sqrt{n}\rfloor+1}^{\lfloor\frac{n-4}{2}\rfloor}\frac{V_{g_{1},i+3}x_{0}^{i+1}}{(2\pi^{2})^{3g_{1}+i+1}(i+4)^{\frac{1}{2}+\varepsilon}i!(n+1)^{\frac{5g_{1}}{2}}}\\
 & \ll_{g}\frac{1}{n^{\frac{1}{2}(1-\varepsilon)}}.
\end{align*}
The last bound follows from the fact that for $i\geqslant\sqrt{n}$,
Theorem \ref{thm:Manin-Zograf} provides
\begin{align*}
\frac{V_{g_{1},i+3}x_{0}^{i+1}}{(2\pi^{2})^{3g_{1}+i+1}(i+4)^{\frac{1}{2}+\varepsilon}i!(n+1)^{\frac{5g_{1}}{2}}} & \ll_{g}\frac{1}{(i+4)^{1+\varepsilon}}.
\end{align*}
\end{proof}
\begin{lem}
\label{lem:top-terms}For any $g\geqslant0$ and $k\leqslant\frac{1}{8}\log(n)$,
the following estimates hold as $n\to\infty$, 
\begin{align*}
\sum_{i=0}^{\lfloor\frac{n-4}{2}\rfloor}{n-3-k \choose i+1}\frac{V_{g_{1},i+3}V_{g-g_{1},n-i-1}}{V_{g,n}} & =\begin{cases}
\frac{j_{0}^{2}}{8\pi^{2}}+O_{g}\left(\frac{1}{n^{\frac{1}{4}}}\right) & \text{if \ensuremath{g_{1}=0,}}\\
O_{g}\left(\frac{1}{n^{\frac{1}{4}}}\right) & \text{if \ensuremath{1\leq g_{1}\leq g,}}
\end{cases}\\
\sum_{i=0}^{\lfloor\sqrt{n}\rfloor}{n-3-k \choose i+1}\frac{V_{g_{1},i+3}V_{g-g_{1},n-i-1}}{V_{g,n}} & =\begin{cases}
\frac{j_{0}^{2}}{8\pi^{2}}+O_{g}\left(\frac{1}{n^{\frac{1}{4}}}\right) & \text{if \ensuremath{g_{1}=0,}}\\
O_{g}\left(\frac{1}{n^{\frac{5}{4}}}\right) & \text{if \ensuremath{1\leq g_{1}\leq g,}}
\end{cases}
\end{align*}

\[
\sum_{i=\lfloor\sqrt{n}\rfloor+1}^{\lfloor\frac{n-4}{2}\rfloor}{n-3-k \choose i+1}\frac{V_{g_{1},i+3}V_{g-g_{1},n-i-1}}{V_{g,n}}\ll_{g}\frac{1}{n^{\frac{1}{4}}}.
\]
\end{lem}

\begin{proof}
We start with $g_{1}=0$. By Lemma \ref{lem:Manin-Zograf-series},
we have 
\[
\sum_{i=0}^{\infty}\frac{V_{0,i+3}}{(i+1)!}\left(\frac{x_{0}}{2\pi^{2}}\right)^{i+1}=\frac{j_{0}^{2}}{8\pi^{2}},
\]
 and so we write
\begin{align*}
 & \left|\sum_{i=0}^{\lfloor\frac{n-4}{2}\rfloor}{n-3-k \choose i+1}\frac{V_{0,i+3}V_{g,n-i-1}}{V_{g,n}}-\sum_{i=0}^{\infty}\frac{V_{0,i+3}}{(i+1)!}\left(\frac{x_{0}}{2\pi^{2}}\right)^{i+1}\right|\\
\leqslant & \underbrace{\left|\sum_{i=0}^{\lfloor\sqrt{n}\rfloor}{n-3-k \choose i+1}\frac{V_{0,i+3}V_{g,n-i-1}}{V_{g,n}}-\sum_{i=0}^{\lfloor\sqrt{n}\rfloor}\frac{V_{0,i+3}}{(2\pi^{2})^{i+1}(i+1)!}x_{0}^{i+1}\right|}_{(1)}\\
 & +\underbrace{\left|\sum_{i=\lfloor\sqrt{n}\rfloor+1}^{\lfloor\frac{n-4}{2}\rfloor}{n-3-k \choose i+1}\frac{V_{0,i+3}V_{g,n-i-1}}{V_{g,n}}\right|}_{(2)}+\underbrace{\left|\sum_{i=\lfloor\sqrt{n}\rfloor+1}^{\infty}\frac{V_{0,i+3}}{(2\pi^{2})^{i+1}(i+1)!}x_{0}^{i+1}\right|}_{(3)}.
\end{align*}
For (1), we use Theorem \ref{thm:Manin-Zograf} to see that because
$i\leq\sqrt{n}$,
\begin{align*}
 & {n-3-k \choose i+1}\frac{V_{g,n-i-1}}{V_{g,n}}\\
= & \frac{(n-3-k)!}{(i+1)!(n-k-i-4)!}\frac{(n-i-1)!}{n!}\left(\frac{n-i}{n+1}\right)^{\frac{5g-7}{2}}\left(\frac{x_{0}}{2\pi^{2}}\right)^{i+1}\left(1+O_{g}\left(\frac{1}{n}\right)\right)\\
= & \left(\prod_{p=0}^{k+2}\left(1-\frac{i+1}{n-p}\right)\right)\frac{1}{(i+1)!}\left(\frac{x_{0}}{2\pi^{2}}\right)^{i+1}\left(1+O_{g}\left(\frac{1}{\sqrt{n}}\right)\right)\\
= & \frac{1}{(i+1)!}\left(\frac{x_{0}}{2\pi^{2}}\right)^{i+1}\left(1+O_{g}\left(\frac{1}{\sqrt{n}}\right)\right),
\end{align*}
and for (2), when $\sqrt{n}<i\leq\left\lfloor \frac{n-4}{2}\right\rfloor $,
\[
{n-3-k \choose i+1}\frac{V_{g,n-i-1}}{V_{g,n}}\leqslant2^{\frac{7}{2}}\frac{1}{(i+1)!}\left(\frac{x_{0}}{2\pi^{2}}\right)^{i+1}\left(1+O_{g}\left(\frac{1}{n}\right)\right).
\]
It follows by Lemma \ref{lem:Manin-Zograf-series} that
\begin{align*}
\left|\sum_{i=0}^{\lfloor\sqrt{n}\rfloor}{n-3-k \choose i+1}\frac{V_{0,i+3}V_{g,n-i-1}}{V_{g,n}}-\sum_{i=0}^{\lfloor\sqrt{n}\rfloor}\frac{V_{0,i+3}}{(2\pi^{2})^{i+1}(i+1)!}x_{0}^{i+1}\right| & \ll_{g}\frac{1}{\sqrt{n}}\sum_{i=0}^{\lfloor\sqrt{n}\rfloor}\frac{V_{i+3}}{(2\pi^{2})^{i+1}(i+1)!}x_{0}^{i+1}\\
 & \ll_{g}\frac{1}{\sqrt{n}},\\
\left|\sum_{i=\lfloor\sqrt{n}\rfloor+1}^{\lfloor\frac{n-4}{2}\rfloor}{n-3-k \choose i+1}\frac{V_{0,i+3}V_{g,n-i-1}}{V_{g,n}}\right| & \ll_{g}\frac{1}{\sqrt{n}},\\
\left|\sum_{i=\lfloor\sqrt{n}\rfloor+1}^{\infty}\frac{V_{0,i+3}}{(2\pi^{2})^{i+1}(i+1)!}x_{0}^{i+1}\right| & \ll_{g}\frac{1}{\sqrt{n}}.
\end{align*}
For $1\leq g_{1}\leq g$, we have
\begin{align*}
 & {n-3-k \choose i+1}\frac{V_{g-g_{1},n-i-1}}{V_{g,n}}\\
= & \frac{(n-3-k)!}{(i+1)!(n-k-i-4)!}\frac{(n-i-1)!}{n!}\left(\frac{n-i}{n+1}\right)^{\frac{5g_{1}-7}{2}}\frac{1}{(n+1)^{\frac{5(g-g_{1})}{2}}}\left(\frac{x_{0}}{2\pi^{2}}\right)^{i+1}\left(\frac{B_{g-g_{1}}}{B_{g}}+O_{g}\left(\frac{1}{n}\right)\right)\\
\ll_{g} & \frac{1}{(i+1)!(n+1)^{\frac{5(g-g_{1})}{2}}}\left(\frac{x_{0}}{2\pi^{2}}\right)^{i+1}.
\end{align*}
Thus, by Lemma \ref{lem:Manin-Zograf-series}

\begin{align*}
\sum_{i=0}^{\lfloor\sqrt{n}\rfloor}{n-3-k \choose i+1}\frac{V_{g-g_{1},i+3}V_{g_{1},n-i-1}}{V_{g,n}} & \ll_{g}\sum_{i=0}^{\lfloor\sqrt{n}\rfloor}\frac{(i+4)^{\frac{5(g-g_{1})}{2}}}{(n+1)^{\frac{5(g-g_{1})}{2}}}\frac{V_{g-g_{1},i+3}}{(i+1)!(i+4)^{\frac{5(g-g_{1})}{2}}}\left(\frac{x_{0}}{2\pi^{2}}\right)^{i+1}\\
\ll_{g}\frac{1}{(n+1)^{\frac{5(g-g_{1})}{4}}} & \sum_{i=0}^{\infty}\frac{V_{g-g_{1},i+3}}{(i+1)!(i+4)^{\frac{5(g-g_{1})}{2}}}\left(\frac{x_{0}}{2\pi^{2}}\right)^{i+1}\ll_{g}\frac{1}{n^{\frac{5}{4}}},
\end{align*}
and 
\begin{align*}
\sum_{i=\left\lfloor \sqrt{n}\right\rfloor +1}^{\lfloor\frac{n-4}{2}\rfloor}{n-3-k \choose i+1}\frac{V_{g-g_{1},i+3}V_{g_{1},n-i-1}}{V_{g,n}} & \ll_{g}\frac{1}{(n+1)^{\frac{5(g-g_{1})}{2}}}\sum_{i=\left\lfloor \sqrt{n}\right\rfloor +1}^{\lfloor\frac{n-4}{2}\rfloor}\frac{V_{g-g_{1},i+3}}{(i+1)!}\left(\frac{x_{0}}{2\pi^{2}}\right)^{i+1}\\
\ll_{g}\sum_{i=\left\lfloor \sqrt{n}\right\rfloor +1}^{\infty} & \frac{V_{g-g_{1},i+3}}{(i+1)!(i+4)^{\frac{5(g-g_{1})}{2}}}\left(\frac{x_{0}}{2\pi^{2}}\right)^{i+1}\ll_{g}\frac{1}{n^{\frac{1}{4}}}.
\end{align*}
The latter claims follow from using the respective estimates already
calculated.
\end{proof}
\begin{rem}
\label{rem:small-sum}We note from the proof of Lemma \ref{lem:top-terms}
that we implicitly also prove the estimate
\[
\sum_{i=0}^{n}\frac{1}{(n+1)^{\frac{5g}{2}}}\frac{V_{g,i+3}}{(i+1)!}\left(\frac{x_{0}}{2\pi^{2}}\right)^{i+1}\ll_{g}\frac{1}{n^{\frac{1}{4}}},
\]
whenever $g\geq1$.
\end{rem}

\section{Asymptotics of Intersection Numbers\protect\label{sec:Asymptotics-of-Intersection}}

In this section we prove Theorems \ref{thm:int-asymptotic} and \ref{thm:I0-asymptotic}. 

\subsection{A special case}

We begin with Theorem \ref{thm:int-asymptotic}, starting with with
the special case of $g=0$, $k=1$ and $d_{1}=1$.
\begin{prop}
\label{prop:tau_1-asymp}As $n\to\infty,$
\[
\frac{[\tau_{0}^{n-1}\tau_{1}]_{0,n}}{V_{0,n}}=12\sum_{i=0}^{\infty}\frac{V_{0,i+3}x_{0}^{i+1}}{(2\pi^{2})^{i+1}(i+1)!}+O\left(\frac{1}{n^{\frac{1}{4}}}\right)=\frac{3j_{0}^{2}}{2\pi^{2}}+O\left(\frac{1}{n^{\frac{1}{4}}}\right).
\]
\end{prop}

\begin{proof}
By Lemma \ref{lem:Mirzakhani-Zograf-(Ib)} we have 
\begin{align*}
[\tau_{0}^{n-1}\tau_{1}]_{0,n} & =12\sum_{\substack{I\sqcup J=\{1,...,n-3\}}
}[\tau_{0}^{|I|+3}]_{0,|I|+3}[\tau_{0}^{|J|+2}]_{0,|J|+2}\\
 & =12\sum_{i=0}^{n-4}{n-3 \choose i}V_{i+3}V_{n-i-1}\\
 & =12\sum_{i=0}^{\lfloor\frac{n-4}{2}\rfloor}{n-3 \choose i}V_{0,i+3}V_{0,n-i-1}+12\sum_{i=0}^{\lfloor\frac{n-4}{2}\rfloor-\delta_{n\in2\mathbb{Z}}}{n-3 \choose i+1}V_{0,i+3}V_{0,n-i-1},
\end{align*}
and the result follows from Lemmas \ref{lem:bottom-terms} and \ref{lem:top-terms}
with $g=0$.
\end{proof}
We next continue with the case of $g=0,k=1$ and induct.
\begin{lem}
\label{lem:rec_rel}For $\ell\geq0$, there exist constants $C_{\ell}>0$
such that
\[
\frac{[\tau_{0}^{n-1}\tau_{\ell}]_{0,n}}{V_{0,n}}=C_{\ell}+O\left(\frac{\ell}{n^{\frac{1}{4}}}\right),
\]
where the implied constant is independent of $\ell$. We have that
$C_{0}=1$ and for $\ell\geq1$,
\[
C_{\ell+1}=8C_{\ell}\sum_{i=0}^{\infty}\frac{V_{0,i+3}}{(2\pi^{2})^{i+1}(i+1)!}x_{0}^{i+1}+4\sum_{i=0}^{\infty}\frac{[\tau_{0}^{i+2}\tau_{\ell}]_{0,i+3}}{(2\pi^{2})^{i+1}(i+1)!}x_{0}^{i+1}.
\]
\end{lem}

\begin{proof}
We proceed by induction on $\ell$. We clearly have $C_{0}=1$ and
Proposition \ref{prop:tau_1-asymp} gives $C_{1}=12\sum_{i=0}^{\infty}\frac{V_{i+3}}{(i+1)!}x_{0}^{i+1}$.
Assuming the result for some $\ell\geq1$, we use Lemma \ref{lem:Mirzakhani-Zograf-(Ib)}
to compute
\begin{align*}
 & \frac{[\tau_{0}^{n-1}\tau_{\ell+1}]_{0,n}}{V_{0,n}}\\
= & 8\sum_{\substack{i=1}
}^{n-3}{n-3 \choose i}\frac{V_{0,i+2}[\tau_{0}^{n-i-1}\tau_{\ell}]_{0,n-i}}{V_{0,n}}+4\sum_{\substack{i=1}
}^{n-3}{n-3 \choose i}\frac{[\tau_{0}^{i+1}\tau_{\ell}]_{0,i+2}V_{0,n-i}}{V_{n}}.\\
= & 8\left(\sum_{i=0}^{\lfloor\frac{n-4}{2}\rfloor}{n-3 \choose i}\frac{V_{0,n-i-1}[\tau_{0}^{i+2}\tau_{\ell}]_{0,i+3}}{V_{0,n}}+\sum_{i=0}^{\lfloor\frac{n-4}{2}\rfloor-\delta_{n\in2\mathbb{Z}}}{n-3 \choose i+1}\frac{V_{0,i+3}[\tau_{0}^{n-i-2}\tau_{\ell}]_{0,n-i-1}}{V_{0,n}}\right)\\
+ & 4\left(\sum_{i=0}^{\lfloor\frac{n-4}{2}\rfloor}{n-3 \choose i}\frac{V_{0,i+3}[\tau_{0}^{n-i-2}\tau_{\ell}]_{0,n-i-1}}{V_{0,n}}+\sum_{i=0}^{\lfloor\frac{n-4}{2}\rfloor-\delta_{n\in2\mathbb{Z}}}{n-3 \choose i+1}\frac{V_{0,n-i-1}[\tau_{0}^{i+2}\tau_{\ell}]_{0,i+3}}{V_{0,n}}\right).
\end{align*}
Now, $[\tau_{0}^{i+2}\tau_{\ell}]_{0,i+3}\leq[\tau_{0}^{i+3}]_{0,i+3}=V_{0,i+3}$
and so by Lemma \ref{lem:bottom-terms},
\[
\sum_{i=0}^{\lfloor\frac{n-4}{2}\rfloor}{n-3 \choose i}\frac{V_{0,n-i-1}[\tau_{0}^{i+2}\tau_{\ell}]_{0,i+3}}{V_{0,n}}=O\left(\frac{1}{n^{\frac{1}{4}}}\right).
\]
Moreover, 
\[
\left|\sum_{i=0}^{\lfloor\frac{n-4}{2}\rfloor-\delta_{n\in2\mathbb{Z}}}{n-3 \choose i+1}\frac{V_{0,i+3}[\tau_{0}^{n-i-2}\tau_{\ell}]_{0,n-i-1}}{V_{0,n}}-C_{\ell}\sum_{i=0}^{\infty}\frac{V_{0,i+3}}{(2\pi^{2})^{i+1}(i+1)!}x_{0}^{i+1}\right|=O\left(\frac{\ell+C_{\ell}+1}{n^{\frac{1}{4}}}\right),
\]
with the implied constant independent of $\ell$. To see this, we
write 
\begin{align*}
 & \left|\sum_{i=0}^{\lfloor\frac{n-4}{2}\rfloor-\delta_{n\in2\mathbb{Z}}}{n-3 \choose i+1}\frac{V_{0,i+3}[\tau_{0}^{n-i-2}\tau_{\ell}]_{0,n-i-1}}{V_{0,n}}-C_{\ell}\sum_{i=0}^{\infty}\frac{V_{0,i+3}}{(2\pi^{2})^{i+1}(i+1)!}x_{0}^{i+1}\right|\\
\leqslant & \left|\sum_{i=0}^{\lfloor\sqrt{n}\rfloor}{n-3 \choose i+1}\frac{V_{0,i+3}[\tau_{0}^{n-i-2}\tau_{\ell}]_{0,n-i-1}}{V_{0,n}}-C_{\ell}\sum_{i=0}^{\lfloor\sqrt{n}\rfloor}\frac{V_{0,i+3}}{(2\pi^{2})^{i+1}(i+1)!}x_{0}^{i+1}\right|\\
+ & C_{\ell}\sum_{i=\lfloor\sqrt{n}\rfloor+1}^{\infty}\frac{V_{0,i+3}}{(2\pi^{2})^{i+1}(i+1)!}x_{0}^{i+1}+\sum_{i=\lfloor\sqrt{n}\rfloor+1}^{\lfloor\frac{n-4}{2}\rfloor-\delta_{n\in2\mathbb{Z}}}{n-3 \choose i+1}\frac{V_{0,i+3}[\tau_{0}^{n-i-2}\tau_{\ell}]_{0,n-i-1}}{V_{0,n}}.
\end{align*}
Then, the latter two terms are $O\left(\frac{C_{\ell}}{n^{\frac{1}{4}}}\right)$
and $O\left(\frac{1}{n^{\frac{1}{4}}}\right)$ respectively by Lemmas
\ref{lem:Manin-Zograf-series} and \ref{lem:top-terms} after bounding
$[\tau_{0}^{n-i-2}\tau_{\ell}]_{0,n-i-1}\leq V_{0,n-i-1}$, where
both implied constants are independent of $\ell$. For the first term,
we use the inductive hypothesis to write $[\tau_{0}^{n-i-2}\tau_{\ell}]_{0,n-i-1}=V_{0,n-i-1}\left(C_{\ell}+O\left(\frac{\ell}{n^{\frac{1}{4}}}\right)\right)$
for $i\leq\lfloor\sqrt{n}\rfloor$ which, combined with the proof
of Lemma \ref{lem:top-terms}, results in 
\begin{align*}
 & \left|\sum_{i=0}^{\lfloor\sqrt{n}\rfloor}{n-3 \choose i+1}\frac{V_{0,i+3}[\tau_{0}^{n-i-2}\tau_{\ell}]_{0,n-i-1}}{V_{0,n}}-C_{\ell}\sum_{i=0}^{\lfloor\sqrt{n}\rfloor}\frac{V_{0,i+3}}{(2\pi^{2})^{i+1}(i+1)!}x_{0}^{i+1}\right|\\
\leqslant & C_{\ell}\left|\sum_{i=0}^{\lfloor\sqrt{n}\rfloor}{n-3 \choose i+1}\frac{V_{0,i+3}V_{0,n-i-1}}{V_{0,n}}-\sum_{i=0}^{\lfloor\sqrt{n}\rfloor}\frac{V_{0,i+3}}{(2\pi^{2})^{i+1}(i+1)!}x_{0}^{i+1}\right|\\
 & +O\left(\frac{\ell}{n^{\frac{1}{4}}}\sum_{i=0}^{\lfloor\sqrt{n}\rfloor}{n-3 \choose i+1}\frac{V_{0,i+3}V_{n-i-1}}{V_{0,n}}\right)=O\left(\frac{C_{\ell}+\ell}{n^{\frac{1}{4}}}\right),
\end{align*}
where the implied constant is independent of $\ell$. We remark that
due to the bound $\left[\tau_{d_{1}},\dots,\tau_{d_{n}}\right]_{g,n}\leqslant V_{g,n}$,
we automatically get $C_{\ell}\leqslant1$ within the inductive hypothesis.
By again using $[\tau_{0}^{n-i-2}\tau_{\ell}]_{0,n-i-1}\leq[\tau_{0}^{n-i-1}]_{0,n-i-1}=V_{0,n-i-1}$
and Lemma \ref{lem:bottom-terms} we obtain
\[
\sum_{i=0}^{\lfloor\frac{n-4}{2}\rfloor}{n-3 \choose i}\frac{V_{0,i+3}[\tau_{0}^{n-i-2}\tau_{\ell}]_{0,n-i-1}}{V_{0,n}}=O\left(\frac{1}{n^{\frac{1}{4}}}\right),
\]
where the implied constant is independent of $\ell$, and using arguments
identical to the proof of Lemma \ref{lem:top-terms} (replacing the
role of $V_{0,i+3}$ with $[\tau_{0}^{i+2}\tau_{\ell}]_{0,i+3}$)
we obtain
\[
\sum_{i=0}^{\lfloor\frac{n-4}{2}\rfloor-\delta_{n\in2\mathbb{Z}}}{n-3 \choose i+1}\frac{V_{0,n-i-1}[\tau_{0}^{i+2}\tau_{\ell}]_{0,i+3}}{V_{n}}=\sum_{i=0}^{\infty}\frac{[\tau_{0}^{i+2}\tau_{\ell}]_{0,i+3}}{(2\pi^{2})^{i+1}(i+1)!}x_{0}^{i+1}+O\left(\frac{1}{n^{\frac{1}{4}}}\right),
\]
where again the implied constant does not depend on $\ell$. It follows
that 
\begin{equation}
\frac{[\tau_{0}^{n-1}\tau_{\ell+1}]_{n}}{V_{0,n}}=8C_{\ell}\sum_{i=0}^{\infty}\frac{V_{0,i+3}}{(2\pi^{2})^{i+1}(i+1)!}x_{0}^{i+1}+4\sum_{i=0}^{\infty}\frac{[\tau_{0}^{i+2}\tau_{\ell}]_{0,i+3}}{(2\pi^{2})^{i+1}(i+1)!}x_{0}^{i+1}+O\left(\frac{\ell+1}{n^{\frac{1}{4}}}\right),\label{eq:final-eq}
\end{equation}
where, since the number of $O\left(\frac{1}{n^{\frac{1}{4}}}\right)$
and $O\left(\frac{\ell+1}{n^{\frac{1}{4}}}\right)$ terms are independent
of $\ell$, the implied constant in (\ref{eq:final-eq}) is independent
of $\ell$, as required.
\end{proof}

\subsection{Computing the constants $C_{\ell}$}

We now seek to compute the constants $C_{\ell}$ using the recursive
relation of Lemma \ref{lem:rec_rel}. Consider the generating function
\[
\Phi_{\ell}(x)=\sum_{i=3}^{\infty}\frac{[\tau_{0}^{i-1}\tau_{\ell}]_{0,i}}{(2\pi^{2})^{i-2}i!}x^{i},
\]
so that
\[
\Phi_{\ell}^{''}(x)=\sum_{i=0}^{\infty}\frac{[\tau_{0}^{i+2}\tau_{\ell}]_{0,i+3}}{(2\pi^{2})^{i+1}(i+1)!}x^{i+1},\hspace{1cm}\Phi_{0}^{''}(x)=\sum_{i=0}^{\infty}\frac{V_{0,i+3}}{(2\pi^{2})^{i+1}(i+1)!}x^{i+1}.
\]
Lemma \ref{lem:rec_rel} motivates us to consider the recurrence relation
\[
\begin{cases}
C_{\ell}(x)=8C_{\ell-1}(x)\Phi_{0}^{''}(x)+4\Phi_{\ell-1}^{''}(x) & \ell\geq2,\\
C_{0}(x)=1,C_{1}(x)=12\Phi_{0}^{''}(x).
\end{cases}
\]
Introduce the function 
\[
F(x,y)=\sum_{m=0}^{\infty}C_{m}(x)y^{m},
\]
so that
\[
F(x,y)-8\Phi_{0}^{''}(x)yF(x,y)=1+4\sum_{m=1}^{\infty}\Phi_{m-1}^{''}(x)y^{m},
\]
and
\[
F(x,y)=\frac{1}{1-8\Phi_{0}^{''}(x)y}\left(1+4\sum_{m=1}^{\infty}\Phi_{m-1}^{''}(x)y^{m}\right).
\]
Differentiating $\ell$ times with respect to $y$ and evaluating
at $y=0$ gives the left hand side equal to $\ell!C_{\ell}$ and the
right hand side can be computed. Indeed, the quotient rule gives
\[
\frac{\partial^{\ell}}{\partial y^{\ell}}\left(\frac{f(x,y)}{g(x,y)}\right)=\sum_{j=0}^{\ell}(-1)^{j}{\ell \choose j}\frac{\partial^{\ell-j}f}{\partial y^{\ell-j}}(x,y)\frac{E_{j}(x,y)}{(g(x,y))^{j+1}},
\]
where 
\[
E_{j}(x,y)=j\frac{\partial g}{\partial y}(x,y)E_{j-1}(x,y)-g(x,y)\frac{\partial E_{j-1}}{\partial y}(x,y).
\]
Applying this in our case, we compute 
\begin{align*}
\frac{\partial^{\ell-j}}{\partial y^{\ell-j}}\left(1+4\sum_{m=1}^{\infty}\Phi_{m-1}^{''}(x)y^{m}\right)(x,0) & =\begin{cases}
1 & j=\ell,\\
4(\ell-j)!\Phi_{\ell-j-1}^{''}(x) & j=0,\ldots,\ell-1,
\end{cases}\\
(g(x,0))^{j+1} & =1,\\
E_{j}(x,0) & =(-8\Phi_{0}^{''}(x))^{j}j!.
\end{align*}
So that formally,
\begin{equation}
C_{\ell}(x)=4\sum_{j=0}^{\ell-1}(8\Phi_{0}^{''}(x))^{j}\Phi_{\ell-j-1}^{''}(x)+(8\Phi_{0}^{''}(x))^{\ell}.\label{eq:C_l-formal}
\end{equation}
Thus we can understand the constants $C_{\ell}$ if we can compute
the second derivative of the generating functions $\Phi_{\ell}(x)$
at $x_{0}$. We observe in the following lemma, that the recursive
formula Lemma \ref{lem:rec_rel} implies that the generating functions
$\Phi_{\ell}$ satisfy a family of ODE's.
\begin{lem}
\label{lem:phi_k-rec}For $\ell\geqslant1$, $x\leqslant x_{0}$,
\[
\Phi_{\ell}^{'''}(x)=8\Phi_{0}^{''}(x)\Phi_{\ell-1}^{'''}(x)+4\Phi_{0}^{'''}(x)\Phi_{\ell-1}^{''}(x).
\]
\end{lem}

\begin{proof}
Recalling that 
\[
\Phi_{\ell}(x)=\sum_{i=3}^{\infty}\frac{[\tau_{0}^{i-1}\tau_{\ell}]_{0,i}}{\left(2\pi^{2}\right)^{i-2}i!}x^{i},
\]
Lemma \ref{lem:rec_rel} implies that 
\begin{align*}
\Phi_{\ell}(x)= & \sum_{i=3}^{\infty}\frac{x^{i}}{\left(2\pi^{2}\right)^{i-2}i!}\Big(8\sum_{j=1}^{i-3}{i-3 \choose j}[\tau_{0}^{j+2}\tau_{\ell-1}]_{0,j+3}[\tau_{0}^{i-j-1}]_{0,i-j-1}\\
 & \thinspace\thinspace\thinspace\thinspace\thinspace\thinspace\thinspace\thinspace\thinspace\thinspace\thinspace\thinspace\thinspace\thinspace\thinspace\thinspace\thinspace\thinspace\thinspace\thinspace\thinspace\thinspace\thinspace\thinspace\thinspace\thinspace\thinspace\thinspace\thinspace\thinspace\thinspace\thinspace\thinspace\thinspace\thinspace\thinspace\thinspace\thinspace\thinspace\thinspace\thinspace\thinspace\thinspace\thinspace\thinspace\thinspace\thinspace\thinspace\thinspace\thinspace\thinspace\thinspace\thinspace\thinspace\thinspace\thinspace\thinspace\thinspace\thinspace\thinspace\thinspace\thinspace\thinspace\thinspace\thinspace\thinspace\thinspace\thinspace\thinspace\thinspace\thinspace\thinspace\thinspace\thinspace\thinspace\thinspace\thinspace\thinspace\thinspace+4\sum_{j=1}^{i-3}{i-3 \choose j}[\tau_{0}^{j+1}\tau_{\ell-1}]_{0,j+2}[\tau_{0}^{i-j}]_{0,i-j}\Big)\\
= & \sum_{i=3}^{\infty}\frac{x^{i}}{\left(2\pi^{2}\right)^{i-2}i(i-1)(i-2)}\Big(8\sum_{j=1}^{i-3}\frac{[\tau_{0}^{j+2}\tau_{\ell-1}]_{0,j+3}}{j!}\frac{[\tau_{0}^{i-j-1}]_{0,i-j-1}}{\left(i-3-j\right)!}\\
 & \thinspace\thinspace\thinspace\thinspace\thinspace\thinspace\thinspace\thinspace\thinspace\thinspace\thinspace\thinspace\thinspace\thinspace\thinspace\thinspace\thinspace\thinspace\thinspace\thinspace\thinspace\thinspace\thinspace\thinspace\thinspace\thinspace\thinspace\thinspace\thinspace\thinspace\thinspace\thinspace\thinspace\thinspace\thinspace\thinspace\thinspace\thinspace\thinspace\thinspace\thinspace\thinspace\thinspace\thinspace\thinspace\thinspace\thinspace\thinspace\thinspace\thinspace\thinspace\thinspace\thinspace\thinspace\thinspace\thinspace\thinspace\thinspace\thinspace\thinspace\thinspace\thinspace\thinspace\thinspace\thinspace\thinspace\thinspace\thinspace\thinspace\thinspace\thinspace\thinspace\thinspace\thinspace\thinspace\thinspace\thinspace\thinspace\thinspace\thinspace\thinspace\thinspace\thinspace\thinspace\thinspace\thinspace\thinspace+4\sum_{j=1}^{i-3}\frac{[\tau_{0}^{j+1}\tau_{\ell-1}]_{0,j+2}}{j!}\frac{[\tau_{0}^{i-j}]_{0,i-j}}{\left(i-3-j\right)!}\Big).
\end{align*}
Then differentiating three times, we obtain
\begin{align*}
\Phi_{\ell}^{'''}(x)= & \sum_{i=3}^{\infty}\frac{x^{i-3}}{\left(2\pi^{2}\right)^{i-2}}\Big(8\sum_{j=1}^{i-3}\frac{[\tau_{0}^{j+2}\tau_{\ell-1}]_{0,j+3}}{j!}\frac{[\tau_{0}^{i-j-1}]_{0,i-j-1}}{\left(i-3-j\right)!}\\
 & \thinspace\thinspace\thinspace\thinspace\thinspace\thinspace\thinspace\thinspace\thinspace\thinspace\thinspace\thinspace\thinspace\thinspace\thinspace\thinspace\thinspace\thinspace\thinspace\thinspace\thinspace\thinspace\thinspace\thinspace\thinspace\thinspace\thinspace\thinspace\thinspace\thinspace\thinspace\thinspace\thinspace\thinspace\thinspace\thinspace\thinspace\thinspace\thinspace\thinspace\thinspace\thinspace\thinspace\thinspace\thinspace\thinspace\thinspace\thinspace\thinspace\thinspace\thinspace\thinspace\thinspace\thinspace\thinspace\thinspace\thinspace\thinspace\thinspace\thinspace\thinspace\thinspace\thinspace\thinspace\thinspace\thinspace\thinspace\thinspace\thinspace\thinspace\thinspace\thinspace\thinspace\thinspace\thinspace\thinspace\thinspace\thinspace\thinspace\thinspace\thinspace\thinspace\thinspace\thinspace\thinspace+4\sum_{j=1}^{i-3}\frac{[\tau_{0}^{j+1}\tau_{\ell-1}]_{0,j+2}}{j!}\frac{[\tau_{0}^{i-j}]_{0,i-j}}{\left(i-3-j\right)!}\Big)\\
= & 8\left(\sum_{n=0}^{\infty}\frac{[\tau_{0}^{n+2}\tau_{\ell-1}]_{0,n+3}}{\left(2\pi^{2}\right)^{n+1}n!}x^{n}\right)\left(\sum_{n=0}^{\infty}\frac{[\tau_{0}^{n+2}]_{0,n+2}}{\left(2\pi^{2}\right)^{n}n!}x^{n}\right)\\
 & \thinspace\thinspace\thinspace\thinspace\thinspace\thinspace\thinspace\thinspace\thinspace\thinspace\thinspace\thinspace\thinspace\thinspace\thinspace\thinspace\thinspace\thinspace\thinspace\thinspace\thinspace\thinspace\thinspace\thinspace\thinspace\thinspace\thinspace\thinspace\thinspace\thinspace\thinspace\thinspace\thinspace\thinspace\thinspace\thinspace\thinspace\thinspace\thinspace\thinspace\thinspace\thinspace\thinspace\thinspace\thinspace\thinspace\thinspace\thinspace\thinspace\thinspace\thinspace\thinspace\thinspace\thinspace\thinspace\thinspace\thinspace\thinspace\thinspace\thinspace\thinspace\thinspace\thinspace\thinspace\thinspace\thinspace\thinspace\thinspace\thinspace\thinspace\thinspace\thinspace\thinspace+4\left(\sum_{n=0}^{\infty}\frac{[\tau_{0}^{n+1}\tau_{\ell-1}]_{0,n+2}}{\left(2\pi^{2}\right)^{n}n!}x^{n}\right)\left(\sum_{n=0}^{\infty}\frac{[\tau_{0}^{n+3}]_{0,n+3}}{\left(2\pi^{2}\right)^{n+1}n!}x^{n}\right).
\end{align*}
It follows that 
\[
\Phi_{\ell}^{'''}(x)=8\Phi_{0}^{''}(x)\Phi_{\ell-1}^{'''}(x)+4\Phi_{0}^{'''}(x)\Phi_{\ell-1}^{''}(x),
\]
as claimed.
\end{proof}
\begin{lem}
\label{lem:phi-eval}For $\ell\geqslant1$ and $x\leqslant x_{0}$,
\[
\Phi_{\ell}^{''}(x)=\frac{2^{3\ell+1}}{\sqrt{\pi}}\frac{\Gamma(\ell+\frac{3}{2})}{\Gamma(\ell+2)}\Phi_{0}^{''}(x)^{\ell+1}.
\]
\end{lem}

\begin{proof}
Using the recurrence relation
\[
\Phi_{\ell}^{'''}(x)=8\Phi_{0}^{''}(x)\Phi_{\ell-1}^{'''}(x)+4\Phi_{0}^{'''}(x)\Phi_{\ell-1}^{''}(x)
\]
from Lemma \ref{lem:phi_k-rec} repeatedly we obtain
\begin{align*}
\Phi_{\ell}^{'''}(x)= & 4\frac{d}{dx}\left(\Phi_{0}^{''}(x)\Phi_{\ell-1}^{''}(x)\right)+4\Phi_{0}^{''}(x)\Phi_{\ell-1}^{'''}(x)\\
= & 4\frac{d}{dx}\left(\Phi_{0}^{''}(x)\Phi_{\ell-1}^{''}(x)\right)+4\left[8\Phi_{0}^{''}(x)^{2}\Phi_{\ell-2}^{'''}(x)+4\Phi_{0}^{'''}(x)\Phi_{0}^{''}(x)\Phi_{\ell-2}^{''}(x)\right]\\
= & 4\frac{d}{dx}\left(\Phi_{0}^{''}(x)\Phi_{\ell-1}^{''}(x)\right)+4\cdot\frac{8}{2}\frac{d}{dx}\left(\Phi_{0}^{''}(x)^{2}\Phi_{\ell-2}^{''}(x)\right)\\
 & +4\left(8-\frac{4}{2}\right)\left[8\Phi_{0}^{''}(x)^{3}\Phi_{\ell-2}^{'''}(x)+4\Phi_{0}^{''}(x)^{2}\Phi_{0}^{'''}(x)\Phi_{\ell-2}^{''}(x)\right]\\
= & \dots\\
= & \sum_{j=1}^{\ell}\left(\prod_{i=1}^{j-1}\left(8-\frac{4}{i}\right)\right)\frac{4}{j}\frac{d}{dx}\left(\Phi_{0}^{''}(x)^{j}\Phi_{\ell-j}^{''}(x)\right).
\end{align*}
Then integrating from $0$ to $x$ and using that $\Phi_{\ell}^{''}(0)=0$
for every $\ell$ gives that 
\begin{equation}
\Phi_{\ell}^{''}(x)=\sum_{j=1}^{\ell}\left(\prod_{i=1}^{j-1}\left(8-\frac{4}{i}\right)\right)\frac{4}{j}\Phi_{0}^{''}(x)^{j}\Phi_{\ell-j}^{''}(x).\label{eq:A-l-recurs}
\end{equation}
It follows from the recursive formula (\ref{eq:A-l-recurs}) that
there for every $l\geqslant0$ is an $A_{l}$ such that
\[
\Phi_{\ell}^{''}(x)=A_{\ell}\Phi_{0}^{''}(x)^{\ell+1}.
\]
Inserting this back into the recursive relation we obtain 
\[
A_{\ell}(\ell+1)\Phi_{0}^{''}(x)^{\ell}\Phi_{0}^{'''}(x)=8\ell A_{\ell-1}\Phi_{0}^{''}(x)^{\ell}\Phi_{0}^{'''}(x)+4A_{\ell-1}\Phi_{0}^{'''}(x)\Phi_{0}^{''}(x)^{\ell}.
\]
After cancelling, we then see that 
\[
A_{\ell}=\frac{4(2\ell+1)}{\ell+1}A_{\ell-1},
\]
and we know that $A_{0}=1,$ so that
\[
A_{\ell}=\prod_{m=1}^{\ell}\frac{4(2m+1)}{m+1}=\frac{2^{3\ell+1}}{\sqrt{\pi}}\frac{\Gamma(\ell+\frac{3}{2})}{\Gamma(\ell+2)}.
\]
\end{proof}
\begin{lem}
\label{lem:C_l}For each $\ell\geqslant1$, 
\[
C_{\ell}=\left(\frac{j_{0}}{\pi}\right)^{2\ell}\frac{(2\ell+1)}{\sqrt{\pi}}\frac{\Gamma(\ell+\frac{1}{2})}{\Gamma(\ell+1)}.
\]
\end{lem}

\begin{proof}
By (\ref{eq:C_l-formal}) we have 
\[
C_{\ell}=4\sum_{j=0}^{\ell-1}(8\Phi_{0}^{''}(x_{0}))^{j}\Phi_{\ell-j-1}^{''}(x_{0})+(8\Phi_{0}^{''}(x_{0}))^{\ell},
\]
and so from Lemma \ref{lem:phi-eval}
\[
C_{\ell}=8^{\ell}\Phi_{0}^{''}(x_{0})^{\ell}\left(\frac{1}{\sqrt{\pi}}\sum_{j=0}^{\ell-1}\frac{\Gamma(\ell-j+\frac{1}{2})}{\Gamma(\ell-j+1)}+1\right)=8^{\ell}\Phi_{0}^{''}(x_{0})^{\ell}\frac{(2\ell+1)}{\sqrt{\pi}}\frac{\Gamma(\ell+\frac{1}{2})}{\Gamma(\ell+1)}.
\]
The result then follows from the fact that $\Phi_{0}^{''}(x_{0})=\frac{j_{0}^{2}}{8\pi^{2}}$
from Lemma \ref{lem:Manin-Zograf-series}.
\end{proof}

\subsection{Proof of Theorems \ref{thm:I0-asymptotic} and \ref{thm:int-asymptotic}}

To prove the remainder of the intersection number asymptotics, we
will again work inductively using the relations of Lemma \ref{lem:rec_rel}
and the computation of the constants $C_{\ell}$. 
\begin{proof}[Proof of Theorem \ref{thm:int-asymptotic}.]
 We proceed by induction on $k$. When $k=0$, the result is obvious
so let us assume that the result is true for some $1\leq k<\frac{1}{8}\log(n)$.
We wish to show that for any $\ell\leq3g+n-3-\sum_{i=1}^{k}d_{i}$
that 
\[
\frac{[\tau_{0}^{2}\tau_{0}^{n-(k+1)-2}\tau_{\ell}\tau_{d_{1}}\cdots\tau_{d_{k}}]_{g,n}}{V_{g,n}}=C_{\ell}\prod_{i=1}^{k}C_{d_{i}}+O_{g}\left(\frac{\sum_{i=1}^{k}d_{i}+\ell}{n^{\frac{1}{4}}}\right),
\]
where the implied constant is independent of $d_{1},\dots,d_{k},\ell$.
In the following, implied constants in $O$'s will be independent
of $n,d_{1},\dots,d_{k}$ and $\ell$ unless specified otherwise.
We shall induct on $\ell$ and note that the base case of $\ell=0$
is simply the inductive hypothesis for the induction in $k$ as $C_{0}=1$.
Let $d'=(\underbrace{0,\ldots,0}_{n-(k+1)-2},d_{1},\ldots,d_{k})$
so that by Lemma \ref{lem:Mirzakhani-Zograf-(Ib)},
\begin{align*}
\frac{[\tau_{0}^{2}\tau_{0}^{n-(k+1)-2}\tau_{\ell}\tau_{d_{1}}\cdots\tau_{d_{k}}]_{g,n}}{V_{g,n}}= & \underbrace{\frac{[\tau_{0}^{4}\tau_{\ell}\prod_{i=1}^{k}\tau_{d_{i}}]_{g-1,n+2}}{V_{g,n}}}_{(1)}\\
 & +\frac{8}{V_{g,n}}\underbrace{\sum_{\substack{\substack{I\sqcup J=\{1,...,n-3\}\\
g_{1}+g_{2}=g
}
}
}[\tau_{0}^{2}\tau_{\ell-1}\prod_{i\in I}\tau_{d_{i}'}]_{g_{1},|I|+3}[\tau_{0}^{2}\prod_{i\in J}\tau_{d_{i}'}]_{g_{2},|J|+2}}_{(2)}\\
 & +\frac{4}{V_{g,n}}\underbrace{\sum_{\substack{\substack{I\sqcup J=\{1,...,n-3\}\\
g_{1}+g_{2}=g
}
}
}[\tau_{0}\tau_{\ell-1}\prod_{i\in I}\tau_{d_{i}'}]_{g_{1},|I|+2}[\tau_{0}^{3}\prod_{i\in J}\tau_{d_{i}'}]_{g_{2},|J|+3}}_{(3)}.
\end{align*}
Term (1) can be controlled easily using Theorem \ref{thm:Manin-Zograf}
since
\begin{align}
\frac{[\tau_{0}^{4}\tau_{\ell}\prod_{i=1}^{k}\tau_{d_{i}}]_{g-1,n+2}}{V_{g,n}} & \leq\frac{V_{g-1,n+2}}{V_{g,n}}=\frac{(n+2)!}{n!}\frac{(n+3)^{\frac{5(g-1)-7}{2}}}{(n+1)^{\frac{5g-7}{2}}}x_{0}^{-2}(2\pi^{2})^{-1}\left(1+O_{g}\left(\frac{1}{n}\right)\right)\ll_{g}\frac{1}{n^{\frac{1}{2}}}.\label{eq:lower-genus}
\end{align}
Next we bound (2). Consider the summands where $|I|\leq\left\lfloor \frac{n-4}{2}\right\rfloor $
for which we bound both of the intersection numbers by the respective
moduli space volumes 
\[
[\tau_{0}^{2}\tau_{\ell-1}\prod_{i\in I}\tau_{d_{i}'}]_{g_{1},|I|+3}[\tau_{0}^{2}\prod_{i\in J}\tau_{d_{i}'}]_{g_{2},|J|+2}\leq V_{g_{1},|I|+3}V_{g_{2},|J|+2}.
\]
The contribution of these summands is thus bounded by
\[
\sum_{g_{1}=0}^{g}\sum_{i=0}^{\left\lfloor \frac{n-4}{2}\right\rfloor }{n-3 \choose i}\frac{V_{g_{1},i+3}V_{g-g_{1},n-i-1}}{V_{g,n}}=O_{g}\left(\frac{1}{n^{\frac{1}{4}}}\right),
\]
using Lemma \ref{lem:bottom-terms}. For the summands where $|I|>\left\lfloor \frac{n-4}{2}\right\rfloor $,
we split into two cases: 
\begin{enumerate}
\item $n-(k+1)-1,\ldots,n-3\in I$ .
\item At least one of the indices $n-(k+1)-1,\ldots,n-3$ does not lie in
$I$.
\end{enumerate}
For summands where the first case occurs, we have by the inductive
hypothesis in $\ell$ that
\begin{align*}
[\tau_{0}^{2}\tau_{\ell-1}\prod_{i\in I}\tau_{d_{i}'}]_{g_{1},|I|+3} & =[\tau_{0}^{2+|I|-k}\tau_{\ell-1}\prod_{i=1}^{k}\tau_{d_{i}}]_{g_{1},|I|+3}\\
 & =V_{g_{1},|I|+3}\left(C_{\ell-1}\prod_{i=1}^{k}C_{d_{i}}+O_{g}\left(\frac{\sum_{i=1}^{k}d_{i}+\ell-1}{n^{\frac{1}{4}}}\right)\right),
\end{align*}
as $|I|>\left\lfloor \frac{n-4}{2}\right\rfloor $, where the implied
constant is independent of $d_{1},\dots,d_{k},\ell$, and 
\[
[\tau_{0}^{2}\prod_{i\in J}\tau_{d_{i}'}]_{g_{2},|J|+2}=[\tau_{0}^{|J|+2}]_{g_{2},|J|+2}=V_{g_{2},|J|+2}.
\]
 Thus, the contribution of these summands is given by 
\[
8\sum_{g_{1}=0}^{g}\sum_{i=\left\lfloor \frac{n-4}{2}\right\rfloor +1}^{n-4}{n-3-k \choose i-k}\frac{V_{g_{1},i+3}V_{g-g_{1},n-i-1}}{V_{g,n}}\left(C_{\ell-1}\prod_{i=1}^{k}C_{d_{i}}+O_{g}\left(\frac{\sum_{i=1}^{k}d_{i}+\ell-1}{n^{\frac{1}{4}}}\right)\right),
\]
since after assigning the indices $n-(k+1)-1,\ldots,n-3$ to $I$
there are $i-k$ elements of $I$ left to choose and $n-3-k$ elements
to choose them from. Re-indexing this summation we obtain 
\[
8\sum_{g_{1}=0}^{g}\sum_{i=0}^{\left\lfloor \frac{n-4}{2}\right\rfloor -\delta_{n\in2\mathbb{Z}}}{n-3-k \choose i+1}\frac{V_{g-g_{1},i+3}V_{g_{1},n-i-1}}{V_{g,n}}\left(C_{\ell-1}\prod_{i=1}^{k}C_{d_{i}}+O_{g}\left(\frac{\sum_{i=1}^{k}d_{i}+\ell-1}{n^{\frac{1}{4}}}\right)\right).
\]
which is equal to $8\Phi_{0}^{''}(x_{0})C_{\ell-1}\prod_{i=1}^{k}C_{d_{i}}+O_{g}\left(\frac{\sum_{i=1}^{k}d_{i}+\ell}{n^{\frac{1}{4}}}\right)$.
To see this, we note that by Lemma \ref{lem:top-terms}, we have 
\begin{align*}
 & \left|\sum_{i=0}^{\left\lfloor \frac{n-4}{2}\right\rfloor -\delta_{n\in2\mathbb{Z}}}{n-3-k \choose i+1}\frac{V_{0,i+3}V_{g,n-i-1}}{V_{g,n}}C_{\ell-1}\prod_{i=1}^{k}C_{d_{i}}-\Phi_{0}^{''}(x_{0})C_{\ell-1}\prod_{i=1}^{k}C_{d_{i}}\right|\\
\leqslant & \delta_{n\in2\mathbb{Z}}{n-3-k \choose \frac{n}{2}-1}\frac{V_{0,\frac{n}{2}+1}V_{g,\frac{n}{2}+1}}{V_{g,n}}C_{\ell-1}\prod_{i=1}^{k}C_{d_{i}}+O_{g}\left(\frac{\sum_{i=1}^{k}d_{i}+\ell-1}{n^{\frac{1}{4}}}\right).
\end{align*}
 By Theorem \ref{thm:Manin-Zograf} we have 
\begin{align*}
 & {n-3-k \choose \frac{n}{2}-1}\frac{V_{0,\frac{n}{2}+1}V_{g,\frac{n}{2}+1}}{V_{g,n}}\\
= & \frac{(n-3-k)!}{\left(\frac{n}{2}-1\right)!\left(\frac{n}{2}-2-k\right)!}\frac{\left(\frac{n}{2}+1\right)!^{2}}{n!}\left(\frac{n}{2}+2\right)^{\frac{5g-14}{2}}(n+1)^{-\frac{5g-7}{2}}x_{0}^{-2}(2\pi^{2})^{-1}\left(B_{0}+O_{g}\left(\frac{1}{n}\right)\right)=O_{g}\left(\frac{1}{n^{\frac{3}{2}}}\right).
\end{align*}
Using that each $C_{d_{i}}\leqslant1$ (which we get for free by the
fact that $\left[\prod\tau_{d_{i}}\right]_{g,n}\leqslant V_{g,n}$),
the $g_{1}=g$ term is equal to $8\Phi_{0}^{''}(x_{0})C_{\ell-1}\prod_{i=1}^{k}C_{d_{i}}+O_{g}\left(\frac{\sum_{i=1}^{k}d_{i}+\ell}{n^{\frac{1}{4}}}\right)$.
For $0\leq g_{1}<g$, we use Lemma \ref{lem:top-terms} to obtain
\begin{align*}
 & \sum_{g_{1}=0}^{g-1}\sum_{i=0}^{\left\lfloor \frac{n-4}{2}\right\rfloor -\delta_{n\in2\mathbb{Z}}}{n-3-k \choose i+1}\frac{V_{g-g_{1},i+3}V_{g_{1},n-i-1}}{V_{g,n}}\left(C_{\ell-1}\prod_{i=1}^{k}C_{d_{i}}+O_{g}\left(\frac{\sum_{i=1}^{k}d_{i}+\ell-1}{n^{\frac{1}{4}}}\right)\right)\\
 & \ll_{g}\sum_{g_{1}=0}^{g-1}\frac{1}{n^{\frac{1}{4}}}\left(C_{\ell-1}\prod_{i=1}^{k}C_{d_{i}}+O_{g}\left(\frac{\sum_{i=1}^{k}d_{i}+\ell-1}{n^{\frac{1}{4}}}\right)\right)\ll_{g}\frac{\sum_{i=1}^{k}d_{i}+\ell}{n^{\frac{1}{4}}},
\end{align*}
and so the claim follows.

Next, we consider the contribution to (2) from the summands where
at least one of the indices $n-(k+1)-1,\ldots,n-3$ does not lie in
$I$ and $|I|>\left\lfloor \frac{n-4}{2}\right\rfloor $. This is
equal to 
\[
8\sum_{a=0}^{k-1}\sum_{g_{1}=0}^{g}\sum_{\substack{I\sqcup J=\{1,...,n-3\}\\
|I|>\left\lfloor \frac{n-4}{2}\right\rfloor \\
I\text{ has exactly \ensuremath{a} indices from \ensuremath{\{n-(k+1)-1,...,n-3\}}}
}
}\frac{[\tau_{0}^{2}\tau_{\ell-1}\prod_{i\in I}\tau_{d_{i}'}]_{g_{1},|I|+3}[\tau_{0}^{2}\prod_{i\in J}\tau_{d_{i}'}]_{g-g_{1},|J|+2}}{V_{g,n}}.
\]
We use the trivial bound $[\tau_{0}^{2}\tau_{\ell-1}\prod_{i\in I}\tau_{d_{i}'}]_{g_{1},|I|+3}[\tau_{0}^{2}\prod_{i\in J}\tau_{d_{i}'}]_{g-g_{1},|J|+2}\leq V_{g_{1},|I|+3}V_{g-g_{1},|J|+2}$
so that the sum has an upper bound of the form 
\begin{align*}
 & 8\sum_{a=0}^{k-1}\sum_{g_{1}=0}^{g}\sum_{\substack{i=\left\lfloor \frac{n-4}{2}\right\rfloor +1}
}^{n-3-(k-a)}{k \choose a}{n-3-k \choose i-a}\frac{V_{g_{1},i+3}V_{g-g_{1},n-i-1}}{V_{g,n}}\\
= & 8\sum_{a=0}^{k-1}\sum_{g_{1}=0}^{g}\sum_{\substack{i=k-a-1}
}^{\left\lfloor \frac{n-4}{2}\right\rfloor -\delta_{n\in2\mathbb{Z}}}{k \choose a}{n-3-k \choose i+1+a-k}\frac{V_{g-g_{1},i+3}V_{g_{1},n-i-1}}{V_{g,n}}.
\end{align*}
Using Theorem \ref{thm:Manin-Zograf} we see that since $k\leq\frac{1}{8}\log(n)$
and $i\leq\left\lfloor \frac{n-4}{2}\right\rfloor $,
\begin{align*}
{n-3-k \choose i+1+a-k}\frac{V_{g_{1},n-i-1}}{V_{g,n}}\leqslant & \frac{1}{(n+1)^{\frac{5(g-g_{1})}{2}}}\frac{2^{\frac{7}{2}}}{n-a-3}\frac{1}{i!}\frac{x_{0}^{i+1}}{\left(2\pi^{2}\right)^{3(g-g_{1})+i+1}}\\
 & \cdot\left(\prod_{p=0}^{a+2}\frac{n-i-(p+1)}{n-p}\right)\left(\prod_{q=1}^{k-a-1}\frac{i-q}{n-(q+a+3)}\right)\left(1+O_{g}\left(\frac{1}{n}\right)\right)\\
\leqslant & \frac{C}{(n+1)^{\frac{5(g-g_{1})}{2}}}\frac{1}{n^{\frac{3}{8}}}\frac{1}{(i+4)^{\frac{5}{8}}}\frac{1}{i!}\frac{x_{0}^{i+1}}{\left(2\pi^{2}\right)^{3(g-g_{1})+i+1}}\left(1+O_{g}\left(\frac{1}{n}\right)\right)
\end{align*}
Thus,
\begin{align*}
 & 8\sum_{a=0}^{k-1}\sum_{g_{1}=0}^{g}\sum_{\substack{i=k-a-1}
}^{\left\lfloor \frac{n-4}{2}\right\rfloor -\delta_{n\in2\mathbb{Z}}}{k \choose a}{n-3-k \choose i+1+a-k}\frac{V_{g-g_{1},i+3}V_{g_{1},n-i-1}}{V_{g,n}}\\
 & \ll_{g}\frac{1}{n^{\frac{3}{8}}}\sum_{g_{1}=0}^{g}\frac{1}{(n+1)^{\frac{5(g-g_{1})}{2}}}\sum_{a=0}^{k-1}{k \choose a}\sum_{i=k-a-1}^{\left\lfloor \frac{n-4}{2}\right\rfloor }(i+4)^{-\frac{5}{8}}\frac{V_{g-g_{1},i+3}}{i!}\frac{x_{0}^{i+1}}{\left(2\pi^{2}\right)^{3(g-g_{1})+i+1}}\\
 & \ll_{g}\frac{g2^{k}}{n^{\frac{3}{8}}}\ll_{g}\frac{1}{n^{\frac{1}{4}}},
\end{align*}
where we use the fact that 
\[
\frac{1}{(n+1)^{\frac{5(g-g_{1})}{2}}}\frac{V_{g-g_{1},i+3}}{i!}\frac{x_{0}^{i+1}}{\left(2\pi^{2}\right)^{3(g-g_{1})+i+1}}\ll_{g}\frac{(i+3)!}{i!}\frac{(i+4)^{\frac{5(g-g_{1})-7}{2}}}{(n+1)^{\frac{5(g-g_{1})}{2}}}\ll_{g}(i+4)^{-\frac{1}{2}}.
\]
In conclusion, we find that
\[
(2)=8\Phi_{0}^{''}(x_{0})C_{\ell-1}\prod_{i=1}^{k}C_{d_{i}}+O\left(\frac{\sum_{i=1}^{k}d_{i}+\ell}{n^{\frac{1}{4}}}\right).
\]
We evaluate (3) in a similar manner. If $|I|>\left\lfloor \frac{n-4}{2}\right\rfloor $,
then we use trivial volume bounds on the intersection numbers and
obtain an upper bound from this contribution of
\begin{align*}
 & \frac{4}{V_{g,n}}\sum_{g_{1}=0}^{g}\sum_{\substack{i=\left\lfloor \frac{n-4}{2}\right\rfloor +1}
}^{n-3}{n-3 \choose i}V_{g_{1},i+2}V_{g-g_{1},n-i}\\
= & \frac{4}{V_{g,n}}\sum_{g_{1}=0}^{g}\sum_{\substack{i=0}
}^{\left\lfloor \frac{n-4}{2}\right\rfloor }{n-3 \choose i}V_{g-g_{1},i+3}V_{g_{1},n-i-1}\ll_{g}\frac{1}{n^{\frac{1}{4}}},
\end{align*}
by Lemma \ref{lem:bottom-terms}. When $|I|\leq\left\lfloor \frac{n-4}{2}\right\rfloor $,
we split between the case where $n-(k+1)-1,\ldots,n-3\in J$ or not.
When they are in $J$ the contribution comes from the term when $g_{1}=0$
and we have by the inductive hypothesis in $k$ that 
\[
[\tau_{0}^{3}\prod_{i\in J}\tau_{d_{i}'}]_{g,|J|+3}=V_{g,|J|+3}\left(\prod_{i=1}^{k}C_{d_{i}}+O_{g}\left(\frac{\sum_{i=1}^{k}d_{i}}{n^{\frac{1}{4}}}\right)\right),
\]
and so the contribution from these terms is 
\begin{align*}
 & \frac{4}{V_{g,n}}\sum_{i=1}^{\left\lfloor \frac{n-4}{2}\right\rfloor }{n-3-k \choose i}[\tau_{0}^{i+1}\tau_{\ell-1}]_{0,i+2}V_{g,n-i}\left(\prod_{i=1}^{k}C_{d_{i}}+O\left(\frac{\sum_{i=1}^{k}d_{i}}{n^{\frac{1}{4}}}\right)\right)\\
 & =\frac{4}{V_{g,n}}\sum_{i=0}^{\left\lfloor \frac{n-4}{2}\right\rfloor -1}{n-3-k \choose i+1}[\tau_{0}^{i+2}\tau_{\ell-1}]_{0,i+3}V_{g,n-i-1}\left(\prod_{i=1}^{k}C_{d_{i}}+O_{g}\left(\frac{\sum_{i=1}^{k}d_{i}}{n^{\frac{1}{4}}}\right)\right).
\end{align*}
Analogous application of Lemma \ref{lem:top-terms} (but replacing
$V_{0,i+3}$ by $[\tau_{0}^{i+2}\tau_{\ell-1}]_{0,i+3}$) results
in this being equal to
\[
4\Phi_{\ell-1}^{''}(x_{0})\prod_{i=1}^{k}C_{d_{i}}+O\left(\frac{\sum_{i=1}^{k}d_{i}+\ell}{n^{\frac{1}{4}}}\right).
\]
The terms for $g_{1}>0$ can easily be shown to be $O_{g}\left(\frac{\sum_{i=1}^{k}d_{i}+\ell}{n^{\frac{1}{4}}}\right)$
identically as for (2) using Lemma \ref{lem:top-terms}. When at least
one of the indices $n-(k+1)-1,\ldots,n-3$ is not in $J$ we bound
the intersection numbers by their corresponding moduli space volumes
to obtain an upper bound on the contribution by

\begin{align*}
 & 4\sum_{a=1}^{k}\sum_{g_{1}=0}^{g}{k \choose a}\sum_{i=a}^{\left\lfloor \frac{n-4}{2}\right\rfloor }{n-3-k \choose i-a}\frac{V_{g_{1},i+2}V_{g-g_{1},n-i}}{V_{g,n}}\\
 & =4\sum_{a=0}^{k-1}\sum_{g_{1}=0}^{g}{k \choose a}\sum_{i=k-a-1}^{\left\lfloor \frac{n-4}{2}\right\rfloor -1}{n-3-k \choose i+1+k-a}\frac{V_{g-g_{1},i+3}V_{g_{1}n-i-1}}{V_{g,n}}\ll_{g}\frac{1}{n^{\frac{1}{4}}},
\end{align*}
as before. This means that 
\[
(3)=4\Phi_{\ell-1}^{''}(x_{0})\prod_{i=1}^{k}C_{d_{i}}+O_{g}\left(\frac{1}{n^{\frac{1}{4}}}\right),
\]
so that 
\begin{align}
\frac{[\tau_{0}^{2}\tau_{0}^{n-(k+1)-2}\tau_{\ell}\tau_{d_{1}}\cdots\tau_{d_{k}}]_{g,n}}{V_{g,n}} & =8\Phi_{0}^{''}(x_{0})C_{\ell-1}\prod_{i=1}^{k}C_{d_{i}}+4\Phi_{\ell-1}^{''}(x_{0})\prod_{i=1}^{k}C_{d_{i}}+O_{g}\left(\frac{\sum_{i=1}^{k}d_{i}+\ell}{n^{\frac{1}{4}}}\right)\nonumber \\
 & =C_{\ell}\prod_{i=1}^{k}C_{d_{i}}+O_{g}\left(\frac{\sum_{i=1}^{k}d_{i}+\ell}{n^{\frac{1}{4}}}\right),\label{eq:final-eq-general}
\end{align}
as required, using the recursive formula from Lemma \ref{lem:rec_rel}
for the final equality. Since the implied constants in the $O_{g}\left(\frac{1}{n^{\frac{1}{4}}}\right)$
and $O_{g}\left(\frac{\sum_{i=1}^{k}d_{i}+\ell-1}{n^{\frac{1}{4}}}\right)$
terms are independent of $d_{1},\dots,d_{k},\ell$, and the number
of such terms we combine together is independent of $d_{1},\dots,d_{k},\ell$
the implied constant in (\ref{eq:final-eq-general}) is also independent
of $d_{1},\dots,d_{k},\ell$.
\end{proof}

\begin{proof}[Proof of Theorem \ref{thm:I0-asymptotic}]
Theorem \ref{thm:I0-asymptotic} follows from Theorem \ref{thm:int-asymptotic}
and Theorem \ref{thm:Mirz-vol-exp} by recognising the Taylor expansion
of $\prod_{i=1}^{k}I_{0}(x_{i})$. We have 
\begin{align*}
\frac{V_{g,n}(\ell_{1},\dots,\ell_{k},0_{n-k})}{V_{g,n}} & =\sum_{|d|\leq3g+n-3}\prod_{i=1}^{k}\left(\frac{j_{0}\ell_{i}}{2\pi}\right)^{2d_{i}}\frac{1}{\sqrt{\pi}}\frac{\Gamma(d_{i}+\frac{1}{2})}{\Gamma(d_{i}+1)}\frac{1}{(2d_{i})!}\\
 & \thinspace\thinspace\thinspace\thinspace\thinspace\thinspace\thinspace\thinspace\thinspace\thinspace\thinspace\thinspace\thinspace\thinspace\thinspace\thinspace\thinspace\thinspace\thinspace\thinspace\thinspace\thinspace\thinspace\thinspace\thinspace\thinspace\thinspace\thinspace\thinspace\thinspace\thinspace\thinspace\thinspace\thinspace\thinspace\thinspace\thinspace\thinspace\thinspace\thinspace\thinspace\thinspace\thinspace\thinspace\thinspace+O\left(\frac{1}{n^{\frac{1}{4}}}\sum_{|d|\leq3g+n-3}|d|\prod_{i=1}^{k}\frac{\left(\frac{\ell_{i}}{2}\right)^{2d_{i}}}{(2d_{i}+1)!}\right)\\
= & \sum_{d_{1}=0}^{3g+n-3}\cdots\sum_{d_{k}=0}^{3g+n-3}\prod_{i=1}^{k}\left(\frac{j_{0}\ell_{i}}{2\pi}\right)^{2d_{i}}\frac{1}{\sqrt{\pi}}\frac{\Gamma(d_{i}+\frac{1}{2})}{\Gamma(d_{i}+1)}\frac{1}{(2d_{i})!}\\
 & \thinspace\thinspace\thinspace\thinspace\thinspace\thinspace\thinspace\thinspace\thinspace\thinspace\thinspace\thinspace-\sum_{3g+n-3<|d|\leq k(3g+n-3)}\prod_{i=1}^{k}\left(\frac{j_{0}\ell_{i}}{2\pi}\right)^{2d_{i}}\frac{1}{\sqrt{\pi}}\frac{\Gamma(d_{i}+\frac{1}{2})}{\Gamma(d_{i}+1)}\frac{1}{(2d_{i})!}\\
 & \thinspace\thinspace\thinspace\thinspace\thinspace\thinspace\thinspace\thinspace\thinspace\thinspace\thinspace\thinspace\thinspace\thinspace\thinspace\thinspace\thinspace\thinspace\thinspace\thinspace\thinspace\thinspace\thinspace\thinspace\thinspace\thinspace\thinspace\thinspace\thinspace\thinspace\thinspace\thinspace\thinspace\thinspace\thinspace\thinspace\thinspace\thinspace\thinspace\thinspace\thinspace\thinspace\thinspace\thinspace\thinspace\thinspace+O\left(\frac{1}{n^{\frac{1}{4}}}\sum_{|d|\leq3g+n-3}|d|\prod_{i=1}^{k}\frac{\left(\frac{\ell_{i}}{2}\right)^{2d_{i}}}{(2d_{i}+1)!}\right).
\end{align*}
But, 
\[
\left(\frac{j_{0}\ell_{i}}{2\pi}\right)^{2d_{i}}\frac{1}{\sqrt{\pi}}\frac{\Gamma(d_{i}+\frac{1}{2})}{\Gamma(d_{i}+1)}\frac{1}{(2d_{i})!}=\left(\frac{j_{0}\ell_{i}}{4\pi}\right)^{2d_{i}}\frac{1}{(d_{i}!)^{2}},
\]
and we recall that 
\[
I_{0}(x)=\sum_{d=0}^{\infty}\left(\frac{x}{2}\right)^{2d}\frac{1}{(d!)^{2}},
\]
so that
\[
\sum_{d_{1}=0}^{3g+n-3}\cdots\sum_{d_{k}=0}^{3g+n-3}\prod_{i=1}^{k}\left(\frac{j_{0}\ell_{i}}{2\pi}\right)^{2d_{i}}\frac{1}{\sqrt{\pi}}\frac{\Gamma(d_{i}+\frac{1}{2})}{\Gamma(d_{i}+1)}\frac{1}{(2d_{i})!}=\prod_{i=1}^{k}\left(I_{0}\left(\frac{j_{0}\ell_{i}}{2\pi}\right)-\sum_{d=3g+n-2}^{\infty}\left(\frac{j_{0}\ell_{i}}{4\pi}\right)^{2d}\frac{1}{(d!)^{2}}\right).
\]
We then notice that 
\[
\sum_{d=3g+n-2}^{\infty}\left(\frac{j_{0}\ell_{i}}{4\pi}\right)^{2d}\frac{1}{(d!)^{2}}\leq\frac{\exp\left(\left(\frac{j_{0}\ell_{i}}{4\pi}\right)^{2}\right)}{(3g+n-2)!}.
\]
 Using also that $I_{0}(x)\leq e^{x}$, we find 
\begin{align*}
 & \prod_{i=1}^{k}\left(I_{0}\left(\frac{j_{0}\ell_{i}}{2\pi}\right)-\sum_{d=3g+n-2}^{\infty}\left(\frac{j_{0}\ell_{i}}{4\pi}\right)^{2d}\frac{1}{(d!)^{2}}\right)\\
= & \prod_{i=1}^{k}I_{0}\left(\frac{j_{0}\ell_{i}}{2\pi}\right)+O\left(\sum_{\substack{I\sqcup J=\{1,\ldots,k\}\\
J\neq\emptyset
}
}\prod_{i\in I}\exp\left(\frac{j_{0}\ell_{i}}{4\pi}\right)\prod_{j\in J}\frac{\exp\left(\left(\frac{j_{0}\ell_{i}}{4\pi}\right)^{2}\right)}{(3g+n-2)!}\right)\\
= & \prod_{i=1}^{k}I_{0}\left(\frac{j_{0}\ell_{i}}{2\pi}\right)+O\left(\sum_{j=1}^{k}{k \choose j}\frac{\exp\left(\frac{j_{0}\sum_{i=1}^{k}\max(\ell_{i},\ell_{i}^{2})}{4\pi}\right)}{(3g+n-2)!^{j}}\right)\\
= & \prod_{i=1}^{k}I_{0}\left(\frac{j_{0}\ell_{i}}{2\pi}\right)+O_{k}\left(\frac{\exp\left(\frac{j_{0}\sum_{i=1}^{k}\max(\ell_{i},\ell_{i}^{2})}{4\pi}\right)}{(3g+n-2)!}\right).
\end{align*}
Using the same estimates, we see that 
\[
\sum_{3g+n-3<|d|\leq k(3g+n-3)}\prod_{i=1}^{k}\left(\frac{j_{0}\ell_{i}}{2\pi}\right)^{2d_{i}}\frac{1}{\sqrt{\pi}}\frac{\Gamma(d_{i}+\frac{1}{2})}{\Gamma(d_{i}+1)}\frac{1}{(2d_{i})!}=O_{k}\left(\frac{\exp\left(\sum_{i=1}^{k}\left(\frac{j_{0}\ell_{i}}{4\pi}\right)^{2}\right)}{(3g+n-2)!^{k}}\right).
\]
For the remaining error, we notice that 
\[
\frac{1}{n^{\frac{1}{4}}}\sum_{|d|\leq3g+n-3}|d|\prod_{i=1}^{k}\frac{\left(\frac{\ell_{i}}{2}\right)^{2d_{i}}}{(2d_{i}+1)!}\leq\frac{1}{n^{\frac{1}{4}}}\sum_{i=1}^{k}l_{i}\frac{\partial}{\partial l_{i}}\left(\prod_{i=1}^{k}\frac{\sinh\left(\frac{\ell_{i}}{2}\right)}{\frac{\ell_{i}}{2}}\right)\leqslant\frac{\prod_{i=1}^{k}\cosh\left(\frac{\ell_{i}}{2}\right)}{n^{\frac{1}{4}}},
\]
to obtain 
\[
\frac{V_{g,n}(\ell_{1},\dots,\ell_{k},0_{n-k})}{V_{g,n}}=\prod_{i=1}^{k}I_{0}\left(\frac{j_{0}\ell_{i}}{2\pi}\right)+O\left(\frac{\prod_{i=1}^{k}\cosh\left(\frac{\ell_{i}}{2}\right)}{n^{\frac{1}{4}}}\right).
\]

\end{proof}

\section{Proof of Theorem \ref{thm:eigenvalues} \protect\label{sec:Short-geodesics-and}}

The purpose of this section is the prove Theorem \ref{thm:eigenvalues}. 
\begin{proof}[Proof of Theorem \ref{thm:eigenvalues}.]
 Let $N_{2}\left(X,L\right)$ denote the number of unoriented, primitive
closed geodesics on $X$ with length $\leqslant L$ that separate
off exactly two cusps and no genus. We first show that for any $L<2\text{arcsinh}1$,
there is a constant $C(L)$ such that
\begin{equation}
\mathbb{P}_{g,n}\left[\frac{N_{2}\left(X,L\right)}{n}<C(L)\right]\to0\label{eq:prob-eq}
\end{equation}
as $n\to\infty$.

By our assumption, any geodesic of length $<L$ is simple and we can
write
\[
N_{2}\left(X,L\right)=\sum_{[\gamma]}\sum_{\alpha\in\text{MCG\ensuremath{\cdot}}[\gamma]}\mathds{1}_{\leqslant L}(\ell_{\alpha}(X)),
\]
where the first summation is over all mapping class group orbits of
(homotopy classes of) simple closed curves which separate $\Sigma_{g,n}$
into $\Sigma_{0,3}$ and $\Sigma_{g,n-1}$. Since by our definition,
the mapping class group respects the labelling of punctures, there
are ${n \choose 2}$ orbits. By Mirzakhani's integration formula,
\begin{align*}
\mathbb{E}\left[\frac{N_{2}\left(X,L\right)}{n}\right] & =\frac{1}{nV_{g,n}}{n \choose 2}\int_{0}^{L}xV_{g,n-1}(0_{n-n_{1}},x)\mathrm{d}x\\
 & =\frac{1}{n}{n \choose 2}\frac{V_{g,n-1}}{V_{g,n}}\int_{0}^{L}x\left(I_{0}\left(\frac{j_{0}x}{2\pi}\right)+O_{g}\left(\frac{1}{n^{\frac{1}{4}}}\cosh\left(\frac{x}{2}\right)\right)\right)\mathrm{d}x\\
 & =\frac{1}{n}{n \choose 2}\frac{V_{g,n-1}}{V_{g,n}}\left(\frac{2\pi L}{j_{0}}I_{1}\left(\frac{j_{0}L}{2\pi}\right)+O_{g,L}\left(\frac{1}{n^{\frac{1}{4}}}\right)\right).
\end{align*}
By Theorem \ref{thm:Manin-Zograf},

\begin{align*}
{n \choose 2}\frac{V_{g,n-1}}{nV_{g,n}} & =\frac{n-1}{n}\frac{1}{2}\left(\frac{n}{n+1}\right)^{\frac{5g-7}{2}}\left(\frac{x_{0}}{2\pi^{2}}\right)\left(1+O_{g}\left(\frac{1}{n}\right)\right)\\
 & =\frac{x_{0}}{4\pi^{2}}\left(1+O_{g}\left(\frac{1}{n^{\frac{1}{2}}}\right)\right).
\end{align*}
Then
\[
\mathbb{E}\left[\frac{N_{2}\left(X,L\right)}{n}\right]=\frac{LJ_{1}\left(j_{0}\right)I_{1}\left(\frac{j_{0}L}{2\pi}\right)}{4\pi}\left(1+O\left(\frac{1}{n^{\frac{1}{4}}}\right)\right).
\]
Now we compute the variance. 
\begin{align*}
\mathbb{E}\left[\left(\frac{N_{2}\left(X,L\right)}{n}\right)^{2}\right] & =\frac{1}{n^{2}}\mathbb{E}\left[\sum_{\substack{\gamma\in\mathcal{P}(X)\\
\gamma\ \text{separates off exactly \ensuremath{2} cusps}
}
}\ind_{\leqslant L}\left(\ell_{\gamma}\left(X\right)\right)\right]\\
 & +\frac{1}{n^{2}}\mathbb{E}\left[\sum_{\substack{(\gamma_{1},\gamma_{2})\in\mathcal{P}\left(X\right)\times\mathcal{P}\left(X\right)\\
\gamma_{1}\neq\gamma_{2}\\
\gamma_{1},\gamma_{2}\ \text{separate off exactly \ensuremath{2} cusps}
}
}\ind_{\leqslant L}\left(\ell_{\gamma_{1}}\left(X\right),\ell_{\gamma_{2}}\left(X\right)\right)\right].
\end{align*}
The first term on the right hand side is just equal to 
\[
\frac{\mathbb{E}\left[N_{2}\left(X,L\right)\right]}{n^{2}}=O_{g,L}\left(\frac{1}{n}\right).
\]
Since $L<2\text{arcsinh}1$, any pair of curves $\gamma_{1}\neq\gamma_{2}$
with length $\leqslant L$ are disjoint by the Collar Lemma (c.f.
\cite[Theorem 4.4.6]{Bu2010}). We calculate the second term as
\begin{align*}
 & \frac{1}{n^{2}}\mathbb{E}\left[\sum_{\substack{(\gamma_{1},\gamma_{2})\in\mathcal{P}\left(X\right)\times\mathcal{P}\left(X\right)\\
\gamma_{1}\neq\gamma_{2}\\
\gamma_{1},\gamma_{2}\ \text{separate off exactly \ensuremath{2} cusps}
}
}\ind_{\frac{1}{L}}\left(\ell_{\gamma_{1}}\left(X\right),\ell_{\gamma_{2}}\left(X\right)\right)\right]\\
 & \,\,\,\,\,\,\,\,\,\,\,\,\,\,\,\,\,\,\,\,\,\,\,\,\,\,\,\,\,\,\,\,\,\,\,\,\,\,={n \choose 2,2}\frac{1}{n^{2}V_{g,n}}\int_{0}^{L}\int_{0}^{L}x_{1}x_{2}V_{g,n-2}(x_{1},x_{2})dx_{1}dx_{2}\\
 & \,\,\,\,\,\,\,\,\,\,\,\,\,\,\,\,\,\,\,\,\,\,\,\,\,\,\,\,\,\,\,\,\,\,\,\,\,\,={n \choose 2,2}\frac{V_{g,n-2}}{n^{2}V_{g,n}}\left(\left(\frac{2\pi L}{j_{0}}I_{1}\left(\frac{j_{0}L}{2\pi}\right)\right)^{2}+O_{g,L}\left(\frac{1}{n^{\frac{1}{4}}}\right)\right).
\end{align*}
Then by Theorem \ref{thm:Manin-Zograf}, 
\begin{align}
 & {n \choose 2,2}\frac{V_{g,n-2}}{n^{2}V_{g,n}}\nonumber \\
 & =\frac{1}{4}\frac{(n-2)(n-3)}{n^{2}}\left(\frac{n-1}{n+1}\right)^{\frac{5g-7}{2}}\left(\frac{x_{0}}{2\pi^{2}}\right)^{2}\left(1+O_{g}\left(\frac{1}{n}\right)\right)\nonumber \\
 & =\frac{1}{4}\left(\frac{x_{0}}{2\pi^{2}}\right)^{2}\left(1+O_{g}\left(\frac{1}{n^{\frac{1}{2}}}\right)\right).\label{eq:n-n1-n2bd}
\end{align}
Then 
\[
\mathrm{Var}\left(\frac{N_{2}\left(X,L\right)}{n}\right)=\mathbb{E}\left[\left(\frac{N_{2}\left(X,L\right)}{n}\right)^{2}\right]-\left(\mathbb{E}\left[\frac{N_{2}\left(X,L\right)}{n}\right]\right)^{2}\ll_{g,L}\frac{1}{n^{\frac{1}{4}}}.
\]
Taking $C(L)$ to be any constant $<\frac{LJ_{1}\left(j_{0}\right)I_{1}\left(\frac{j_{0}L}{2\pi}\right)}{4\pi}$,
the claim follows by applying Chebyshev's inequality.

By applying a min-max argument, it is shown in \cite[Section 5]{Hi.Th.22},
if $N_{2}\left(X,\frac{\ep}{6}\right)>k$ then $\lambda_{k}\leqslant\ep$.
Taking $L=\frac{\ep}{6}$, it follows from (\ref{eq:prob-eq}) that
$\lambda_{C\left(\frac{\ep}{6}\right)n}<\ep$ with probability tending
to $1$ as $n\to\infty$.
\end{proof}

\section{Relative frequencies of closed curves}

The purpose of this section is to prove the following.
\begin{thm}
\label{thm:simple-nonsimple-2}Let $L=L(n)>0$ be any function with
$L\to\infty$ as $n\to\infty$ and $L=O\left(\log n\right)$. Then
there exists a function $\varepsilon(n)$ with $\varepsilon(n)\to0$
as $n\to\infty$ such that a Weil-Petersson random surface $X\in\mathcal{M}_{g,n}$
satisfies
\[
N^{\mathrm{s}}(X,L)\leqslant\varepsilon(n)N^{\mathrm{ns}}(X,L)
\]
with probability tending to $1$ as $n\to\infty$.
\end{thm}

\subsection*{Outline of the proof}

Let $L\to\infty$ as $n\to\infty$ with $L=O\left(\log n\right)$.
First we show, using Theorem \ref{thm:I0-asymptotic}, that there
is a constant $c_{1}<1$ such that with high probability, the number
of simple closed geodesics of length less than $L$ is at most $ne^{c_{1}L}$.
The fact that $c_{1}<1$ is crucial here. We then want a lower bound
for the number of closed geodesics of length up to $L$ which holds
with high probability and grows faster than $ne^{c_{1}L}$ as $n\to\infty$.

To achieve this, we note that by the estimates of Section \ref{sec:Short-geodesics-and},
for any $\ep>0$, with probability tending to $1$ as $n\to\infty$,
a random surface in $\mathcal{M}_{g,n}$ has at least $c\left(\ep\right)n$
closed geodesics of length at most $\ep$ which separate off a pair
of pants with two cusps. This tells us that with high probability
on a random surface there are at least $c\left(\ep\right)n$ disjoint
subsurfaces which are pairs of pants with two cusps and geodesic boundary
with length less than $\ep$. We show in Subsection \ref{subsec:pants-counting}
that one can pick $\ep$ so that there are at least $e^{c_{2}L}$
closed geodesics of length less than $L$ in each subsurface where
$c_{2}$ satisfies $c_{1}<c_{2}<1$. It follows that 
\[
\frac{N^{s}\left(X,L\right)}{N^{\text{ns}}\left(X,L\right)}\leqslant\text{const}\cdot e^{\left(c_{1}-c_{2}\right)L}
\]
with probability tending to $1$ as $n\to\infty$.

\subsection{Growth of the number of simple closed geodesics}
\begin{lem}
\label{lem:simple-curve-upper-bound}Let $L=L\left(n\right)$ be any
function with $L\to\infty$ as $n\to\infty$ and $L=O\left(\log n\right)$.
Then there exists a positive constant $c_{1}$ with $c_{1}<1$ such
that

\[
\mathbb{P}_{g,n}\left[N^{\text{s}}\left(X,L\right)\geqslant ne^{c_{1}L}\right]\to0,
\]
as $n\to\infty$.
\end{lem}

\begin{proof}
Let $\Sigma_{g,n}$ denote a genus-$g$ topological surface with $n$
labelled punctures. We have

\[
N(X,L)=\sum_{[\gamma]}\sum_{\alpha\in[\gamma]}\mathds{1}_{\leqslant L}(\ell_{\alpha}(X)),
\]
where the exterior summation is over all mapping class group orbits
of homotopy classes of simple closed curves on $\Sigma_{g,n}$. On
$\Sigma_{g,n}$ there are the following types of mapping class group
orbits:
\begin{enumerate}
\item When $g\geq1$, there is a single mapping class group orbit of non-separating
curves that when cut along, reduces the genus by one and adds two
boundaries.
\item For each configuration $\left\{ \left(g_{1},n_{1},\left\{ c_{1}^{1},\ldots,c_{n_{1}}^{1}\right\} \right),\left(g_{2},n_{2},\left\{ c_{1}^{2},\ldots,c_{n_{2}}^{2}\right\} \right)\right\} $
where $g_{1}+g_{2}=g$, $n_{1}+n_{2}=n$, $2g_{i}+n_{i}-1>0$ and
$\left\{ c_{1}^{1},\ldots,c_{n_{1}}^{1},c_{1}^{2},\ldots,c_{n_{2}}^{2}\right\} =\left\{ 1,\ldots,n\right\} $,
there is a single mapping class group orbit for the curves that separate
the surface into a component with genus $g_{1}$, $n_{1}$ cusps labelled
with $\{c_{1}^{1},\ldots,c_{n_{1}}^{1}\}$ and one boundary and a
component with genus $g_{2}$, $n_{2}$ cusps labelled with $\{c_{1}^{2},\ldots,c_{n_{2}}^{2}\}$
and one boundary.
\end{enumerate}
We will write $0_{m}:=\underbrace{(0,\ldots,0)}_{m}$. Using Theorem
\ref{thm:MIF}, we see that
\begin{align*}
\mathbb{E}_{g,n}\left(\frac{N\left(X,L\right)}{n}\right) & =\underbrace{\frac{1}{nV_{g,n}}\int_{0}^{L}xV_{g-1,n+2}(0_{n},x,x)\mathrm{d}x}_{(1)}\\
 & +\underbrace{\frac{1}{nV_{g,n}}\sum_{g_{1}=1}^{g-1}\sum_{n_{1}=0}^{\lfloor\frac{n}{2}\rfloor}{n \choose n_{1}}\int_{0}^{L}xV_{g_{1},n_{1}+1}(0_{n_{1}},x)V_{g-g_{1},n-n_{1}+1}(0_{n-n_{1}},x)\mathrm{d}x}_{(2)}\\
 & +\underbrace{\frac{1}{nV_{g,n}}\sum_{n_{1}=2}^{n}{n \choose n_{1}}\int_{0}^{L}xV_{0,n_{1}+1}(0_{n_{1}},x)V_{g,n-n_{1}+1}(0_{n-n_{1}},x)\mathrm{d}x}_{(3)}.
\end{align*}
For term (1), we use the trivial volume bounds Lemma \ref{lem:sinh-upper-bound}
and (\ref{eq:lower-genus}) to obtain
\[
\frac{1}{nV_{g,n}}\int_{0}^{L}xV_{g-1,n+2}(0_{n-2},x,x)\mathrm{d}x=\frac{e^{L}}{n^{\frac{1}{4}}}\frac{V_{g-1,n+2}}{nV_{g,n}}=\frac{e^{L}}{n^{\frac{3}{2}}}.
\]
For term (2), we use the trivial bound 
\[
V_{g_{1},n_{1}+1}(0_{n_{1}},x)V_{g-g_{1},n-n_{1}+1}(0_{n-n_{1}},x)\leq\frac{\sinh^{2}\left(\frac{x}{2}\right)}{\left(\frac{x}{2}\right)^{2}}V_{g_{1},n_{1}+1}V_{g-g_{1},n-n_{1}+1},
\]
 and employ Theorem \ref{thm:Manin-Zograf} to write

\[
{n \choose n_{1}}\frac{V_{g-g_{1},n-n_{1}+1}}{nV_{g,n}}\ll_{g}\frac{1}{n_{1}!}\frac{1}{(n+1)^{\frac{5g_{1}}{2}}}\frac{x_{0}^{n_{1}-1}}{(2\pi^{2})^{3g_{1}+n_{1}-1}}.
\]
Then, 
\begin{align*}
 & \frac{1}{nV_{g,n}}\sum_{g_{1}=1}^{g-1}\sum_{n_{1}=0}^{\lfloor\frac{n}{2}\rfloor}{n \choose n_{1}}\int_{0}^{L}xV_{g_{1},n_{1}+1}(0_{n_{1}},x)V_{g-g_{1},n-n_{1}+1}(0_{n-n_{1}},x)\mathrm{d}x\\
 & \ll_{g}e^{L}\sum_{g_{1}=1}^{g-1}\sum_{n_{1}=0}^{\lfloor\frac{n}{2}\rfloor}\frac{V_{g_{1},n_{1}+1}}{n_{1}!}\frac{1}{(n+1)^{\frac{5g_{1}}{2}}}\frac{x_{0}^{n_{1}-1}}{(2\pi^{2})^{3g_{1}+n_{1}-1}}\ll_{g}\frac{e^{L}}{n^{\frac{1}{4}}},
\end{align*}
by Remark \ref{rem:small-sum}. The leading order contribution comes
from term (3). To see this, write
\begin{align*}
 & \Big|\sum_{n_{1}=2}^{n}{n \choose n_{1}}\int_{0}^{L}x\frac{V_{0,n_{1}+1}(0_{n_{1}},x)V_{g,n-n_{1}+1}(0_{n-n_{1}},x)}{nV_{g,n}}\mathrm{d}x\\
 & \thinspace\thinspace\thinspace\thinspace\thinspace\thinspace\thinspace\thinspace\thinspace\thinspace\thinspace\thinspace\thinspace\thinspace\thinspace\thinspace\thinspace\thinspace\thinspace-\sum_{n_{1}=2}^{\infty}\frac{1}{n_{1}!}\left(\frac{x_{0}}{2\pi^{2}}\right)^{n_{1}-1}\int_{0}^{L}xV_{0,n_{1}+1}(0_{n_{1}},x)I_{0}\left(\frac{j_{0}x}{2\pi}\right)\mathrm{d}x\Big|\\
\leqslant & \underbrace{\Big|\sum_{n_{1}=2}^{\lfloor\sqrt{n}\rfloor}{n \choose n_{1}}\int_{0}^{L}x\frac{V_{0,n_{1}+1}(0_{n_{1}},x)V_{g,n-n_{1}+1}(0_{n-n_{1}},x)}{nV_{g,n}}\mathrm{d}x}\\
 & \underbrace{\,\,\,\,\,\,\,\,\,\,\,\,\thinspace\thinspace\thinspace\thinspace\,-\sum_{n_{1}=2}^{\lfloor\sqrt{n}\rfloor}\frac{1}{n_{1}!}\left(\frac{x_{0}}{2\pi^{2}}\right)^{n_{1}-1}\int_{0}^{L}xV_{0,n_{1}+1}(0_{n_{1}},x)I_{0}\left(\frac{j_{0}x}{2\pi}\right)\mathrm{d}x\Big|}_{(a)}\\
+ & \underbrace{\sum_{n_{1}=\lfloor\sqrt{n}\rfloor}^{n}{n \choose n_{1}}\int_{0}^{L}x\frac{V_{0,n_{1}+1}(0_{n_{1}},x)V_{g,n-n_{1}+1}(0_{n-n_{1}},x)}{nV_{g,n}}\mathrm{d}x}_{(b)}\\
 & \thinspace\thinspace\thinspace\thinspace\thinspace\thinspace+\underbrace{\sum_{n_{1}=\lfloor\sqrt{n}\rfloor}^{\infty}\frac{1}{n_{1}!}\left(\frac{x_{0}}{2\pi^{2}}\right)^{n_{1}-1}\int_{0}^{L}xV_{0,n_{1}+1}(0_{n_{1}},x)I_{0}\left(\frac{j_{0}x}{2\pi}\right)\mathrm{d}x}_{(c)}.
\end{align*}
We start with bounding (a) by using that fact that since $n_{1}\leq\sqrt{n}$,
we have $n-n_{1}+1\geq\frac{n}{2}$ so that by Theorem \ref{thm:I0-asymptotic},
\[
V_{g,n-n_{1}+1}(0_{n-n_{1}},x)=V_{g,n-n_{1}+1}\left(I_{0}\left(\frac{j_{0}x}{2\pi}\right)+O_{g}\left(\frac{1}{n^{\frac{1}{4}}}\cosh\left(\frac{x}{2}\right)\right)\right).
\]
Then by Theorem \ref{thm:Manin-Zograf},

\begin{align*}
{n \choose n_{1}}\frac{V_{g,n-n_{1}+1}}{nV_{g,n}} & =\frac{n-n_{1}+1}{n}\frac{1}{n_{1}!}\left(\frac{n-n_{1}+2}{n+1}\right)^{\frac{5g-7}{2}}\left(\frac{x_{0}}{2\pi^{2}}\right)^{n_{1}-1}\left(1+O_{g}\left(\frac{1}{n}\right)\right)\\
 & =\frac{1}{n_{1}!}\left(\frac{x_{0}}{2\pi^{2}}\right)^{n_{1}-1}\left(1+O_{g}\left(\frac{1}{n^{\frac{1}{2}}}\right)\right).
\end{align*}
Thus using Lemma \ref{lem:sinh-upper-bound},
\begin{align*}
(a)\ll_{g} & \frac{1}{n^{\frac{1}{2}}}\sum_{n_{1}=2}^{\lfloor\sqrt{n}\rfloor}\frac{1}{n_{1}!}\left(\frac{x_{0}}{2\pi^{2}}\right)^{n_{1}-1}\int_{0}^{L}xV_{0,n_{1}+1}(0_{n_{1}},x)I_{0}\left(\frac{j_{0}x}{2\pi}\right)\mathrm{d}x\\
 & +\frac{1}{n^{\frac{1}{4}}}\sum_{n_{1}=2}^{\lfloor\sqrt{n}\rfloor}\frac{1}{n_{1}!}\left(\frac{x_{0}}{2\pi^{2}}\right)^{n_{1}-1}\int_{0}^{L}xV_{0,n_{1}+1}(0_{n_{1}},x)\cosh\left(\frac{x}{2}\right)\mathrm{d}x\\
 & \ll_{g}e^{L}\frac{1}{n^{\frac{1}{4}}}\sum_{n_{1}=2}^{\lfloor\sqrt{n}\rfloor}\frac{V_{0,n_{1}+1}}{n_{1}!}\left(\frac{x_{0}}{2\pi^{2}}\right)^{n_{1}-1}\ll_{g}\frac{e^{L}}{n^{\frac{1}{4}}},
\end{align*}
with the last line following from Lemma \ref{lem:Manin-Zograf-series}.
For (b), we bound
\[
V_{0,n_{1}+1}(0_{n_{1}},x)V_{g,n-n_{1}+1}(0_{n-n_{1}},x)\leq V_{0,n_{1}+1}V_{g,n-n_{1}+1}\frac{\sinh^{2}\left(\frac{x}{2}\right)}{\left(\frac{x}{2}\right)^{2}},
\]
and then split the summation for $\lfloor\sqrt{n}\rfloor\leq n_{1}\leq\lfloor\frac{n}{2}\rfloor$
and $\lfloor\frac{n}{2}\rfloor\leq n_{1}\leq n$. In the first case,
we use Theorem \ref{thm:Manin-Zograf} to bound 
\[
{n \choose n_{1}}\frac{V_{g,n-n_{1}+1}}{nV_{g,n}}\ll_{g}\frac{1}{n_{1}!}\left(\frac{x_{0}}{2\pi^{2}}\right)^{n_{1}-1},
\]
and in the second case,
\[
{n \choose n_{1}}\frac{V_{0,n_{1}+1}}{nV_{g,n}}\ll_{g}\frac{1}{(n-n_{1})!}\frac{x_{0}^{n-n_{1}-1}}{\left(2\pi^{2}\right)^{3g+n-n_{1}-1}}\frac{1}{(n+1)^{\frac{5g}{2}}}.
\]
Thus by Remark \ref{rem:small-sum} and using Lemma \ref{lem:sinh-upper-bound}
for the other volumes, 
\begin{align*}
(b) & \ll_{g}e^{L}\sum_{n_{1}=\lfloor\sqrt{n}\rfloor}^{\lfloor\frac{n}{2}\rfloor}\frac{V_{0,n_{1}+1}}{n_{1}!}\left(\frac{x_{0}}{2\pi^{2}}\right)^{n_{1}-1}+e^{L}\sum_{n_{1}=\lfloor\frac{n}{2}\rfloor}^{n}\frac{V_{g,n-n_{1}+1}}{(n-n_{1})!}\left(\frac{x_{0}}{2\pi^{2}}\right)^{n-n_{1}-1}\frac{1}{(n+1)^{\frac{5g}{2}}}\\
 & \ll_{g}e^{L}\sum_{n_{1}=\lfloor\sqrt{n}\rfloor}^{\lfloor\frac{n}{2}\rfloor}\frac{V_{0,n_{1}+1}}{n_{1}!}\left(\frac{x_{0}}{2\pi^{2}}\right)^{n_{1}-1}+e^{L}\sum_{n_{1}=0}^{n-\left\lfloor \frac{n}{2}\right\rfloor }\frac{V_{g,n_{1}+1}}{n_{1}!}\left(\frac{x_{0}}{2\pi^{2}}\right)^{n_{1}-1}\frac{1}{(n+1)^{\frac{5g}{2}}}\ll_{g}\frac{e^{L}}{n^{\frac{1}{4}}}.
\end{align*}
For (c), we again use
\[
V_{0,n_{1}+1}(0_{n_{1}},x)\leq V_{0,n_{1}+1}\frac{\sinh\left(\frac{x}{2}\right)}{\frac{x}{2}},
\]
and then apply Lemma \ref{lem:Manin-Zograf-series} to obtain 
\[
(c)\ll_{g,L}\sum_{n_{1}=\lfloor\sqrt{n}\rfloor}^{\infty}\frac{V_{0,n_{1}+1}}{n_{1}!}\left(\frac{x_{0}}{2\pi^{2}}\right)^{n_{1}-1}\ll_{g}\frac{e^{L}}{n^{\frac{1}{4}}}.
\]
Combining, we see that $(3)$ is asymptotically bounded above
\[
(3)\ll_{g}\sum_{n_{1}=2}^{\infty}\frac{1}{n_{1}!}\left(\frac{x_{0}}{2\pi^{2}}\right)^{n_{1}-1}\int_{0}^{L}xV_{0,n_{1}+1}(0_{n_{1}},x)I_{0}\left(\frac{j_{0}x}{2\pi}\right)\mathrm{d}x+\frac{e^{L}}{n^{\frac{1}{4}}}.
\]
Applying once more
\[
V_{0,n_{1}+1}(0_{n_{1}},x)\leq V_{0,n_{1}+1}\frac{\sinh\left(\frac{x}{2}\right)}{\frac{x}{2}},
\]
and noting by Lemma \ref{lem:Manin-Zograf-series} that 
\[
\sum_{n_{1}=2}^{\infty}\frac{V_{0,n_{1}+1}}{n_{1}!}\left(\frac{x_{0}}{2\pi^{2}}\right)^{n_{1}-1}\ll1,
\]
 we see that

\[
(3)\ll_{g}\int_{0}^{L}I_{0}\left(\frac{j_{0}x}{2\pi}\right)\sinh\left(\frac{x}{2}\right)dx+\frac{e^{L}}{n^{\frac{1}{4}}}.
\]
Combining the respective bounds on (1), (2) and (3), we obtain,
\[
\mathbb{E}_{g,n}\left(\frac{N^{\text{s}}\left(X,L\right)}{n}\right)\ll_{g}\int_{0}^{L}I_{0}\left(\frac{j_{0}x}{2\pi}\right)\sinh\left(\frac{x}{2}\right)dx+\frac{e^{L}}{n^{\frac{1}{4}}}.
\]
 Since $\frac{j_{0}}{2\pi}<\frac{1}{2}$, there exists $c_{0}<1$
so that 
\[
\mathbb{E}_{g,n}\left(\frac{N^{\text{s}}\left(X,L\right)}{n}\right)\ll_{g}e^{c_{0}L}+\frac{e^{L}}{n^{\frac{1}{4}}}.
\]
Finally, since $L=O\left(\log n\right)$ there exists a $C>0$ such
that $L\leqslant C\log n$ for $n$ sufficiently large. Then taking
$c_{1}>c_{0}$ with $c_{1}<1$ and $1-c_{1}<\frac{1}{5C}$, by Markov's
inequality, 
\[
\mathbb{P}_{g,n}\left[N^{\text{s}}\left(X,L\right)\geqslant ne^{c_{1}L}\right]\ll_{g}e^{(c_{0}-c_{1})L}+\frac{e^{\left(1-c_{1}\right)L}}{n^{\frac{1}{4}}}\ll_{g}e^{(c_{0}-c_{1})L}+n^{C(1-c_{1})-\frac{1}{4}}\to0
\]
as $n\to\infty$. 
\end{proof}

\subsection{Counting geodesics in pants\protect\label{subsec:pants-counting}}

In this subsection we prove the following.
\begin{prop}
\label{thm:Pants-PGT}Let $Y_{\ell}$ be a hyperbolic pair of pants
with two cusps and one geodesic boundary component of length $\ell$.
For any $\kappa>0$ there are constants $\ell_{0}$ and $T_{0}$ such
that for every $Y_{l}$ with $\ell\in[\ell_{0}/2,\ell_{0}]$ and every
$T\geqslant T_{0},$
\[
N\left(Y_{\ell},T\right)\geqslant\frac{1}{10}e^{\left(1-\kappa\right)T}.
\]
\end{prop}

Proposition \ref{thm:Pants-PGT} relies on a prime geodesic theorem
with precise error terms, due to Naud \cite{Na.05}. For a hyperbolic
surface $S$ and $0\leqslant a<b$, let $N\left(S,[a,b]\right)$ denote
the number of closed geodesics on $S$ with lengths in the interval
$[a,b]$.
\begin{thm}[\cite{Na.05}]
\label{thm:Naud-PGT}Let $S_{\ell}$ be an infinite-volume hyperbolic
pair of pants with two cusps and a funnel of width $\ell<1$. Then
for any $0<a\leq b<1$,
\[
N\left(S_{\ell},T\right)=\textup{li}\left(e^{\delta\left(S_{\ell}\right)T}\right)+O\left(e^{\left(\frac{\delta\left(S_{\ell}\right)}{2}+\frac{1}{4}\right)T}\right),
\]
where the implied constant is uniform over $\ell\in[a,b]$, and $\delta(S_{\ell})$
is the Hausdorff dimension of the limit set of a Fuchsian group $\Gamma$
for which $S_{\ell}=\Gamma\backslash\mathbb{H}.$
\end{thm}

\begin{proof}
This is essentially \cite[Theorem 1.2]{Na.05} with some small adaptations
to make the error term uniform in $\ell$ over the compact window.
To this end, note firstly that, for example by \cite[Theorem 3.1]{McM1998},
the Hausdorff dimension, and hence the first resonance $\delta_{0}(S_{\ell})=\delta(S_{\ell})(1-\delta(S_{\ell}))$
is continuous in $\ell$. Moreover, we fix a topological surface $\Sigma$
and can equip each $S_{\ell}$ with a marking $\varphi:\Sigma\to S_{\ell}$
such that for any free homotopy class of closed curve $[\alpha]$
on $\Sigma$, the length $\ell_{[\alpha]}(S_{\ell})$ of the geodesic
representative of $[\varphi(\alpha)]$ on $S_{\ell}$ is continuous
in $\ell$. 

We now follow precisely the notation as in the proof of \cite[Theorem 1.2]{Na.05}.
Note that for fixed $X$ and $Y$ as in the proof, the number of closed
geodesics on $S_{\ell}$ can be bounded uniformly as $\ell$ ranges
over $[a,b]$ since the systole of these surface is uniformly bounded
below by $a$. It follows that in {[}ibid. Equation (2.19){]} the
geometric (right-hand) side of the formula is locally continuous
and hence continuous on $[a,b]$. Coupled with the continuity of the
first resonance, we see that the error term 

\[
\sum_{s\in\mathcal{R}_{S_{\ell}}\backslash\delta_{0}(S_{\ell})}\hat{g}(s)
\]
is also continuous (note that for us, $\mathcal{R}_{S_{\ell}}=\mathcal{R}_{S_{\ell}}^{+}\cup{\delta_{0}(S_{\ell})}$
since by \cite{Ba.Ma.Mo17}, $\delta_{0}\left(S_{\ell}\right)$ is
the only $L^{2}$-eigenvalue). The proof now continues identically
up to equations (2.27) and (2.28) where the bound on the error term
is dependent on $\ell$ in a pointwise manner of the form:
\begin{proof}
\[
\left|\sum_{s\in\mathcal{R}_{S_{\ell}}\backslash\delta_{0}(S_{\ell})}\hat{g}(s)\right|\leq C_{1}(\ell)+C_{2}(\ell)X^{a}Y^{b}+C_{3}(\ell)X^{c}Y^{d}.
\]
But, using the continuity of the left-hand side, the $C_{i}(\ell)$
on the right-hand side can be made uniform over all $\ell\in[a,b]$.
 The proof then continues identically as in \cite[Theorem 1.2]{Na.05}
propagating these uniform bounds throughout the remainder.
\end{proof}
\end{proof}
We now prove Proposition \ref{thm:Pants-PGT}.
\begin{proof}[Proof of Proposition \ref{thm:Pants-PGT}]
 By gluing a funnel of width $\ell_{i}$ to $Y_{\ell}$, we obtain
a complete surface $S_{\ell}$ of infinite area. Since $Y_{\ell}$
is the convex core of $S_{\ell}$, any closed geodesic on $S_{\ell}$
lies in $Y_{\ell}$, and so it remains to prove the statement for
$S_{\ell}$.

Let $\kappa>0$ be given. By e.g. \cite[Theorem 3.1]{McM1998}, the
Hausdorff dimension $\delta\left(S_{\ell}\right)$ is a continuous
function in $\ell$ with $\lim_{\ell\to0}\delta\left(S_{\ell}\right)=1$.
Then we can find an $\ell_{0}>0$ such that 
\[
\delta\left(S_{\ell}\right)>1-\kappa
\]
for all $\ell<\ell_{0}$. By Theorem \ref{thm:Naud-PGT} we have that
\[
N\left(S_{\ell},T\right)\geqslant\text{li}\left(e^{\delta\left(S_{\ell}\right)T}\right)-O\left(e^{\frac{3}{4}T}\right),
\]
where the implied constant is uniform over $\ell\in[\ell_{0}/2,\ell_{0}]$.
Assuming $\kappa<\frac{1}{4}$, the leading term dominates and there
exists a $T_{0}$ such that for every $\ell<\ell_{0},$
\[
N\left(S_{\ell},T\right)\geqslant\frac{1}{10}e^{\left(1-\kappa\right)T}
\]
for all $T>T_{0}$.
\end{proof}

\subsection{Proof of Theorem \ref{thm:simple-nonsimple-2}}

We can now proceed with the proof of Theorem \ref{thm:simple-nonsimple-2}.
\begin{proof}
Let $L\to\infty$ with $L=O\left(\log n\right)$ be given and $c_{1}$
be the constant given by Lemma \ref{lem:simple-curve-upper-bound}.
Choose $\kappa<1-c_{1}$ and let $\ell_{0}\left(\kappa\right)$ be
given as in Proposition \ref{thm:Pants-PGT}. By an identical argument
to that in the proof of Theorem \ref{thm:eigenvalues}, namely the
probabilistic bounds on $N_{2}(X,L)$, there exists a $C=C\left(\ell_{0}\right)>0$
such that a.a.s.
\[
N_{2}\left(X,[\ell_{0}/2,\ell_{0}]\right)\geqslant Cn,
\]
where $N_{2}\left(X,[\ell_{0}/2,\ell_{0}]\right)$ counts the number
of unoriented primitive closed geodesics on $X$ with lengths in $[\ell_{0}/2,\ell_{0}]$
that separate off a pair of pants with two cusps from $X$ when cut
along. Then a.a.s. $X$ contains at least $Cn$ disjoint subsurfaces
$\left\{ Y_{i}\right\} $ with two cusps and a geodesic boundary with
length in $[\ell_{0}/2,\ell_{0}]$. Therefore by Proposition \ref{thm:Pants-PGT},
\[
\#\left\{ \gamma\in\mathcal{P}\left(X\right)\mid\ell_{\gamma}\left(X\right)\leqslant L\right\} \geqslant\sum_{i}\#\left\{ \gamma\in\mathcal{P}\left(Y_{i}\right)\mid\ell_{\gamma}\left(Y_{i}\right)\leqslant L\right\} \geqslant Cne^{\left(1-\kappa\right)L},
\]
a.a.s. where $\mathcal{P}(X)$ is the collection of primitive closed
geodesic on $X$. By Lemma (\ref{lem:simple-curve-upper-bound}),
we have 
\[
\frac{N^{s}\left(X,L\right)}{N\left(X,L\right)}\leqslant\frac{ne^{c_{1}L}}{Cne^{\left(1-\kappa\right)L}},
\]
a.a.s.. Then since $N\left(X,L\right)=N^{\text{s}}\left(X,L\right)+N^{\text{ns}}\left(X,L\right)$
we see
\[
\frac{N^{s}\left(X,L\right)}{N^{\text{ns}}\left(X,L\right)}=\left(\frac{1}{1-\frac{N^{\text{s}}\left(X,L\right)}{N\left(X,L\right)}}\right)\frac{N^{s}\left(X,L\right)}{N\left(X,L\right)}\to0
\]
as $n\to\infty$.
\end{proof}

\section*{Acknowledgments}

JT was supported by funding from the European Research Council (ERC)
under the European Union’s Horizon 2020 research and innovation programme
(grant agreement No 949143). We thank Michael Magee, Alex Wright and
Yunhui Wu for conversations about this work. We also thank Baptiste
Louf for comments on an earlier version of this work. Finally we thank
Timothy Budd and Nicolas Curien for sharing a preliminary version
of their work \cite{Bu.Cu2024} with us for comparing the geometric
results.

\bibliographystyle{plain}
\bibliography{puncturedspheresbib}

\noindent Will Hide, \\
Mathematical Institute,\\
University of Oxford, \\
Andrew Wiles Building, OX2 6GG Oxford,\\
United Kingdom

\noindent\texttt{william.hide@maths.ox.ac.uk}~\\
\texttt{}~\\

\noindent Joe Thomas, \\
Department of Mathematical Sciences,\\
Durham University, \\
Lower Mountjoy, DH1 3LE Durham,\\
United Kingdom

\noindent\texttt{joe.thomas@durham.ac.uk}
\end{document}